\documentclass[pdflatex,sn-mathphys-num]{sn-jnl}

\usepackage{setspace,tikz,xcolor,mathrsfs,listings,multicol,amssymb,amsfonts,amsmath}
\usepackage{pgfplots}
\usepackage{rotating}
\usepackage[vcentermath]{youngtab}
\usepackage{tkz-base}
\usepackage{tkz-euclide}
\usepackage{tikz-cd}
\usepackage{tikz}
\usepackage{comment}
\usepackage{hyperref}
\usepackage{graphicx}
\usepackage{subcaption}

\usepackage{enumerate}
\usepackage{enumitem}
\usepackage{booktabs}
\usepackage{gensymb}
\usepackage[numbers,sort&compress]{natbib}
\usepackage[centertableaux]{ytableau}
\usepackage[all,cmtip]{xy}
\usetikzlibrary{arrows,matrix,math}
\tikzset{tab/.style={matrix of math nodes,column sep=-.35, row sep=-.35,text height=7pt,text width=7pt,align=center,inner sep=2,font=\footnotesize}}

\usepackage{hyperref}
\newcommand\numberthis{\addtocounter{equation}{1}\tag{\theequation}}

\newcommand{\dt}{\mathrm{dt}}

\newcommand{\dx}{\mathrm{dx}}

\newcommand{\dz}{\mathrm{dz}}

\renewcommand{\Re}{\operatorname{Re}}
\renewcommand{\Im}{\operatorname{Im}}
\newcommand{\Pin}{\operatorname{Pin}} 
\newcommand{\Sp}{Sp}
\newcommand{\SO}{SO}
\newcommand{\Or}{O}  

\newcommand{\GL}{GL}

\newcommand{\PP}{\mathbb{P}}
\newcommand{\ZZ}{\mathbb{Z}}

\newcommand{\CC}{\mathbb{C}}

\definecolor{darkred}{rgb}{0.7,0,0} 

\definecolor{UQgold}{RGB}{196, 158, 54} 
\definecolor{UQpurple}{RGB}{73, 7, 94} 

\usepackage{listings}
\lstdefinelanguage{Sage}[]{Python}
{morekeywords={False,sage,True},sensitive=true}
\definecolor{dblackcolor}{rgb}{0.0,0.0,0.0}
\definecolor{dbluecolor}{rgb}{0.01,0.02,0.7}
\definecolor{dgreencolor}{rgb}{0.2,0.4,0.0}
\definecolor{dgraycolor}{rgb}{0.30,0.3,0.30}

\theoremstyle{plain}
\newtheorem{thm}{Theorem}
\newtheorem{lemma}{Lemma}

\newtheorem{prop}[thm]{Proposition}
\newtheorem{cor}{Corollary}
\theoremstyle{definition}

\newtheorem{remark}{Remark}

\numberwithin{equation}{section}

\usepackage[colorinlistoftodos]{todonotes}

\begin{document}
\title[Fluctuations for symplectic groups]{Fluctuations of Young diagrams for symplectic groups and semiclassical orthogonal polynomials}

\author[1,2]{\fnm{Anton} \sur{Nazarov}}\email{antonnaz@gmail.com}
\equalcont{These authors contributed equally to this work.}

\author*[3]{\fnm{Anton} \sur{Selemenchuk}}\email{selemenchyk@icloud.com}
\equalcont{These authors contributed equally to this work.}

\affil[1]{\orgdiv{Department of High Energy and Elementary
    Particle Physics}, \orgname{St.\ Petersburg State University}, \orgaddress{\street{University Embankment, 7/9}, \city{St.\ Petersburg}, \postcode{199034}, \country{Russia}}}

\affil[2]{\orgname{Beijing Institute of Mathematical Sciences and Applications (BIMSA)}, \orgaddress{\street{Street}, \city{Bejing}, \postcode{101408}, \state{State}, \country{People’s Republic of China}}}

\affil*[3]{\orgdiv{Faculty of Mathematics and Computer Science}, \orgname{St.\ Petersburg State University}, \orgaddress{\street{University Embankment, 7/9}, \city{St.\ Petersburg}, \postcode{199034}, \country{Russia}}}

\abstract{
  Consider an $n\times k$ matrix of i.i.d. Bernoulli random numbers
  with $p=1/2$. The dual RSK algorithm gives a bijection of this matrix
  to a pair of Young tableaux of conjugate shape, which is
  manifestation of skew Howe $\GL_{n}\times \GL_{k}$-duality.
  Thus the probability measure on zero-ones matrix leads to the probability
  measure on Young diagrams proportional to the ratio of the dimension of
  $\GL_{n}\times \GL_{k}$-representation and the dimension of the
  exterior algebra $\bigwedge\left(\CC^{n}\otimes\CC^{k}\right)$.

  Similarly, by applying Proctor's algorithm based on Berele's
  modification of the Schensted insertion, we get skew Howe duality
  for the pairs of groups $\Sp_{2n}\times \Sp_{2k}$. In the limit when
  $n,k\to\infty$ $\GL$-case is relatively easily studied by use of
  free-fermionic representation for the correlation kernel. But for
  the symplectic groups there is no convenient free-fermionic
  representation. We use Christoffel transformation to obtain the
  semiclassical orthogonal polynomials for $\Sp_{2n}\times \Sp_{2k}$
  from Krawtchouk polynomials that describe $\GL_{2n}\times\GL_{2k}$
  case. We derive an integral representation for semiclassical
  polynomials. The study of the asymptotic of this integral
  representation gives us the description of the limit shapes
  and fluctuations of the random Young diagrams for symplectic groups.
}


\maketitle

\tableofcontents
\section*{Introduction}
\label{sec:introduction}
Consider an $n\times k$ matrix of i.i.d. Bernoulli random numbers with
$p=1/2$. Dual Robinson-Schensted-Knuth algorithm gives a
bijection of this matrix to a pair of semistandard Young tableaux of
conjugate shape \cite{knuth1970permutations}, which is manifestation
of skew Howe $\GL_{n}\times \GL_{k}$-duality \cite{R_Howe}. Thus the
probability measure on zero-ones matrices leads to the probability
measure on Young diagrams proportional to the ratio of the dimension
of $\GL_{n}\times \GL_{k}$-representation and the dimension of the
exterior algebra $\bigwedge\left(\CC^{n}\otimes\CC^{k}\right)$. If $n$
and $k$ go to infinity with the same rate these random diagrams in French notation,
rotated by $\frac{\pi}{4}$ and scaled by a factor $\frac{1}{n}$,
converge uniformly in probability to a smooth limit shape supported on
an interval. This limit shape and fluctuations around it were
considered in \cite{GTW01,GTW02,GTW02II}.

Random diagram can be seen as a random point configuration on
a lattice, where a point is put in the left boundary of any unit
interval, where the upper boundary of diagram decays. That
is if diagram $\lambda$ has row lengths
$(\lambda_{1},\dots,\lambda_{n})$ the coordinates of the points would
be $\{\lambda_{i}+n-i\}_{i=1}^{n}$. Using Weyl dimension formula for
irreducible representations of general linear groups it is possible to
present this random point ensemble as orthogonal polynomial ensemble
with Krawtchouk polynomials. The asymptotic behavior of Krawtchouk
polynomials was studied extensively through the years, see
\cite{ISMAIL1998121,BKMM03,LI2000155,Wang_Qui_2015} and references
therein. The corresponding determinatal ensemble was considered in
\cite{johansson2002non,Borodin_2003,borodin2007asymptotics,Borodin_2017}. 

Skew Howe duality for other pairs of classical Lie groups can be
obtained by certain modifications of the insertion algorithm
\cite{proctor1993reflection}. For example, by applying Proctor's
algorithm \cite{proctor1993reflection} based on Berele's modification
\cite{berele1986schensted} of the Schensted insertion
\cite{schensted1961longest}, we get skew Howe duality for the pairs of
groups $\Sp_{2n}\times \Sp_{2k}$. The algorithm that proves
$\SO_{2n+1}\times \Pin_{2k}$ duality was proposed by Benkart and
Stroomer \cite{benkart1991tableaux}, and for $\Or_{2n}\times \SO_{2k}$
by Okada \cite{okada1993robinson} and Terada \cite{Terada93}. All
these algorithms give the bijections between zero-one matrices
enumerating basis elements in an exterior algebra and pairs of
corresponding modifications of semistandard Young tableaux for
classical Lie groups, such as King tableaux for symplectic groups
\cite{King76} and or Sundaram tableaux for orthogonal groups
\cite{Sundaram90}. The dualities lead to the multiplicity-free
decompositions of the exterior algebra
$\bigwedge\left(\CC^{2n}\otimes\CC^{k}\right)$ into the direct sums of
irreducible representations of dual groups pairs. The decompositions
give rise to the probability measures on (symplectic or orthogonal)
Young diagrams $\lambda$ given by the ratio of dimension of the irreducible
component to the dimension of the whole exterior algebra. It is possible to write these measures explicitly, as was demonstrated
in \cite{nazarov2024skew}. 

Similarly to $\GL_{n}\times \GL_{k}$-case if we rotate and scale the
diagram and take the limit $n,k\to\infty$ with the same rate, we get
the uniform convergence to the limit shape. Moreover, the limiting
curve is ``half'' of $\GL_{2n}\times\GL_{2k}$ curve and is described
by the same explicit formula, as demonstrated in
\cite{nazarov2024skew} by using the dimension formulas and solving the
variational problems. The study of the fluctuations in these cases
remained an open question. 

In the present paper we consider the $\Sp_{2n}\times \Sp_{2k}$-skew Howe duality. To write the measure explicitly introduce the coordinates
$a_{i}=\lambda_{i} + n - i + 1$, then as explained in the Section \ref{sec:rand-diagr-sympl}, we get
\begin{equation}\label{measure_int}
  \mu_{n,k}(\left\{a_{i}\right\})=\frac{1}{Z_{n,k}}\prod_{i<j}(a_{i}^{2}-a_{j}^{2})^{2}
  \prod_{\ell=1}^{n}a_{\ell}^{2}\binom{2n+2k}{n+k+a_{\ell}}=\frac{1}{Z_{n,k}}\det\left[\mathcal{K}\left(\frac{a_{i}}{n},\frac{a_{j}}{n}\right)\right]_{i,j=1}^{n},
\end{equation}
where $Z_{n,k}$ is the normalization constant, given in \eqref{normalization}. The
measure contains a
Vandermonde determinant written in terms of squares of variables
$a_{i}$. Therefore it can be expressed as the determinant of the Cristoffel-Darboux kernel $\mathcal{K}\left(u,v\right)$ for certain orthogonal polynomials. For brevity we call these polynomials \emph{symplectic} and they are new in this context.

As demonstrated in \cite{NNP20}, symplectic polynomials are obtained from
Krawtchouk polynomials by the ``lifting'' procedure from
\cite{BUHMANN1992117} which is a variant of the QR-algorithm
for the Christoffel transformation of orthognal
polynomials~\cite{galant1971implementation},~\cite{bueno2004darboux}.
The polynomials are therefore discrete semiclassical orthogonal
polynomials \cite{van_assche_2017} and their asymptotic behavior was
not considered previously. We study various asymptotic regimes for the
orthogonal polynomials and obtain the description of local
fluctuations for random symplectic Young diagrams in the bulk regime, given by the next theorem. 
\begin{thm}[Sine Kernel Limit]
\label{thm:sine-limit}
Let $\mu_{n,k}$ be the measure corresponding to the $\Sp_{2n}\times \Sp_{2k}$ skew Howe duality \eqref{measure_int}, and let $\mathcal{K}(u,v)$ be the associated correlation kernel defined above. Define
\[
K= 2n+2k,\quad \text{and assume } n=HK \text{ with fixed } 0<H<\frac{1}{2}.
\]
Then, in the bulk regime, that is deifned as
\[
u=\frac{1}{H}\Bigl(x+\frac{i}{K}\Bigr),\quad v=\frac{1}{H}\Bigl(x+\frac{j}{K}\Bigr), \quad i,j\in \mathbb{Z},\quad x\in (-1/2,1/2),\quad K\to \infty,
\]
the kernel converges pointwisely to the discrete sine kernel:
\[
\lim\limits_{n\to \infty}\frac{\mathcal{K}\left(\frac{1}{H}\big(x+\frac{i}{K}\big),\frac{1}{H}\big(x+\frac{j}{K}\big)\right)}{\mathcal{K}(x/H,x/H)}=\frac{\sin (\pi \rho(x)(i-j))}{\pi \rho(x)(i-j)},
\]
and the limit density is explicitly given by
\[
\rho(x)= \frac{1}{\pi}\cos^{-1}\!\Biggl(\frac{1-4H}{\sqrt{1-4x^2}}\Biggr).
\]
\end{thm}

Theorem~\ref{thm:sine-limit} follows from a steepest descent analysis of the relevant integral representations for the polynomials and shows that the local fluctuations converge to universal sine kernel statistics; the density $\rho(x)$ is consistent with the limit shape analysis in \cite{nazarov2024skew}.

The paper is organized as follows. In Section
\ref{sec:rand-diagr-sympl} we introduce the probability measure on
Young diagrams, discuss sampling of random diagrams and recall the
definition of the correlation kernel. In Section
\ref{sec:family-semicl-orth} we write the correlation kernel in terms
of the orthogonal polynomials that are obtained from the Krawtchouk
polynomials by Christoffel transformation. Then in Section
\ref{sec:asympt-polyn} we use the integral representation for the
polynomials to study their asymptotic behavior. These results are used
to prove the convergence of the correlation kernel to the discrete
sine kernel in Section \ref{sec:limit-behav-corr}. We discuss possible
further results in Conclusion.

\section{Random diagrams for symplectic groups}
\label{sec:rand-diagr-sympl}
Consider the symplectic group $\Sp_{2n}$ and its first fundamental
representation on $\CC^{2n}$. The exterior algebra
$\bigwedge\left(\CC^{2n}\right)$ is a space of a reducible representation of $\Sp_{2n}$. The exterior algebra
$\bigwedge\left(\CC^{2n}\otimes\CC^{k}\right)$ can be seen as $k$-th
tensor power of this representation
$\bigwedge\left(\CC^{2n}\otimes\CC^{k}\right)=
\left(\bigwedge\left(\CC^{2n}\right)\right)^{\otimes k}$ and therefore
has a natural action of $\Sp_{2n}$. On the other hand the same can be
done for $\Sp_{2k}$, moreover, the actions of $\Sp_{2n}$ and
$\Sp_{2k}$ commute on $\bigwedge\left(\CC^{2n}\otimes\CC^{k}\right)$
and we have a strongly multiplicity-free decomposition into
irreducible representations of $\Sp_{2n}\times \Sp_{2k}$~\cite{cheng2012dualities},\cite{goodman2009symmetry} 
\begin{equation}
  \label{eq:sp-howe-duality}
  \bigwedge\left(\CC^{2n}\otimes\CC^{k}\right)=
  \bigoplus_{\lambda\subset k^{n}} V_{\Sp_{2n}}(\lambda)\otimes
  V_{\Sp_{2k}}(\overline\lambda'). 
\end{equation}
Here $\lambda$ is a Young diagram that is confined inside of
$n\times k$ rectangle, $\overline\lambda'$ is its transposed
complement in the rectangle and by $V_{\Sp_{2n}}(\lambda)$ we denote
the irreducible representation corresponding to the diagram $\lambda$.
This is known as skew $(\Sp_{2n},\Sp_{2k})$ Howe duality
\cite{R_Howe}, \cite{howe1995perspectives}. Combinatorially this
duality can be understood in terms of King tableaux that generalize
semistandard Young tableaux to the symplectic group case.

Let us briefly recall that King tableau \cite{King76} for the
symplectic group $\Sp_{2n}$ is Young diagram $\lambda$ filled in a semistandard way with signed
numbers that have order $1<\bar 1<2<\bar 2<\dots<n<\bar n$. In
English convention it means
that numbers are strictly increasing along the columns and
non-decreasing along the rows. An additional condition is imposed --
all the numbers that are $\leq \bar i$ should not appear below row
number $i$. King tableaux of shape $\lambda$ enumerate the basis
vectors in the irreducible representation $V_{\Sp_{2n}}(\lambda)$.

Consider the space $\CC^{n}$ with the basis vectors
$\{e_{i}\}_{i=1}^{n}$. The space $\bigwedge\left(C^{n}\right)$ is
$2^{n}$-dimensional and its basis vectors
$\{e_{i_{1}}\wedge e_{i_{2}}\wedge \dots\}$ are in one-to-one
correspondence to the columns of zeros and ones that encode which
vectors appear in the exterior product. Therefore the basis vector of the space
$\bigwedge\left(\CC^{2n}\otimes\CC^{k}\right) =
\left(\bigwedge\CC^{n}\right)^{\otimes 2k}$ corresponds to the
$n\times 2k$ zero-one matrix of $2k$ columns.

The following algorithm based upon Berele insertion
\cite{berele1986schensted} was proposed by Proctor in
\cite{proctor1993reflection}. It establishes the correspondence between
 zero-one matrices of size $n\times 2k$, that encode the basis vectors in the left hand side
of the decomposition \eqref{eq:sp-howe-duality}, and pairs of King
tableaux $(P,Q)$ of complementary shapes $(\lambda,\overline\lambda')$
inside of $n\times k$ rectangle, that encode the basis vectors in the
right hand side of the decomposition.
\begin{itemize}
\item First, we pair the columns of the
zero-one matrix. On each step of the algorithm we consider the pair
of columns, so total number of algorithm steps is $k$.
\item We start with the empty tableaux $P$ and $Q$. On the first step
  we fill the $n\times 1$ rectangle with numbers and adjoin some of
  the boxes to the tableau $P$ and remaining to the tableau $Q$ as
  explained below.
\item After the step $r-1$ the shapes of $P,Q$ form the
  $n\times (r-1)$ rectangle. We start the step $r$ by shifting the
  tableau $Q$ to the right by one box, so that there appears an empty
  strip between $P$ and $Q$ in $n\times r$ rectangle.
\item From two consecutive columns number $2r, 2r+1$ we produce a sequence
  of numbers to insert on the step $r$ by reading the columns bottom-up as follows.
\item If in row number $i$ we have $(0,0)$ in the columns, we produce the number $i$,
  for $(0,1)$ we insert nothing, for $(1,0)$ we insert $\bar i,i$ and
  for $(1,1)$ we insert $\bar i$.
\item Then we insert numbers one-by-one by using Schensted insertion
  \cite{schensted1961longest} into tableau $P$. If the additional
  condition is broken by an insertion of number $i$ that is bumping
  number $\bar i$ below the row $i$, we need to erase the box where
  $\bar i$ was. Then we shift this erased box to the boundary as in jeu
  de taquin (to preserve the semi-standard condition).
\item After the insertion of all the numbers at step $r$, the erased
  boxes are filled with numbers $\bar r$ and adjoined to $Q$.
\item All the remaining empty boxes between the extended tableaux $P$
  and $Q$ in $n\times r$-rectangle are filled with $r$ and adjoined to
  $Q$ as well. Then we proceed to the step number $(r+1)$.
\end{itemize}

Besides the combinatorial proof of the decomposition
\eqref{eq:sp-howe-duality}, Proctor's algorithm gives us an effective
way to sample random Young diagrams by generating random zero-one
matrices with some measure and then pushing this measure to Young
diagrams.

By considering the dimension of the components in the decomposition
\eqref{eq:sp-howe-duality}, we get the identity
\begin{equation*}
2^{2nk}=\sum_{\lambda\subset k^{n}} \dim V_{\Sp_{2n}}(\lambda)\cdot
\dim V_{\Sp_{2k}}(\overline\lambda'), 
\end{equation*}
which can be used to introduce the probability measure on the diagrams
that appear in the decomposition. The probability of a diagram is the
ratio of the dimension of the irreducible
$\Sp_{2n}\times\Sp_{2k}$-representation corresponding to the diagram
to the dimension of the whole exterior algebra, this measure
corresponds to the uniform measure on zero-one matrices
\begin{equation}
  \label{eq:probability-measure}
  \mu_{n,k}(\lambda)=2^{-2nk} \dim V_{\Sp_{2n}}(\lambda)\cdot
  \dim V_{\Sp_{2k}}(\overline\lambda'). 
\end{equation}
It is possible to write this measure explicitly, as was demonstrated
in \cite{nazarov2024skew}. To do so introduce the coordinates
$a_{i}=\lambda_{i} + n - i + 1$. These coordinates can be interpreted
as the coordinates of unit intervals where the upper boundary of Young
diagram $\lambda$, rotated by $\frac{\pi}{4}$ (``Russian
convention''), is decreasing. The diagram then can be identified with
a particle configuration $\left\{a_{i}\right\}$ on the integer
lattice. For $\dim V_{\Sp_{2n}}(\lambda)$ we can use Weyl dimension
formula. The problem is to write the dimension of
$V_{\Sp_{2k}}(\overline\lambda')$ in terms of the coordinates
$\{a_{i}\}$, which are introduced for the diagram $\lambda$. It can be
done by representing King tableau as a set of non-intersecting paths,
then switching to the dual paths and using
Lindstr\"om--Gessel--Viennot lemma to express the number of
non-intersecting paths configurations as a determinant
\cite{oh2019identities}. Then Dodgson condensation leads to the
explicit formula
\begin{align}
  \notag
  &\mu_{n,k}(\left\{a_{i}\right\})=\\  \label{eq:probability-measure-explicit-formula}
  &=\frac{2^{2n(1-k)}}{\prod_{i<j}(j-i)(2n+2-i-j)}
   \prod_{i<j}(a_{i}^{2}-a_{j}^{2})^{2}
  \prod_{\ell=1}^{n} a_{\ell}^{2} \frac{(2k-1+2\ell)!}{(k+n+a_{\ell})!(k+n-a_{\ell})!}.
\end{align}

From this derivation the measure is determinantal, and the associated point process is therefore a determinantal point process (see \cite{borodin2009determinantalpointprocesses} and references therein). Consequently, the measure can be written as determinant of the correlation kernel,
\begin{equation}
  \label{eq:measure-as-determinant-of-correlation-kernel}
  \mu_{n,k}(\{a_{i}\})=\det\big[\mathcal{K}(a_{i},a_{j})\big]_{i,j=1}^{n}.
\end{equation}
In particular the density of particles is just a one-point correlation
function
\begin{equation}
  \label{eq:density-definition}
  \rho(a)=\mathcal{K}(a,a).
\end{equation}
We are interested in the limit behavior of the probability measure
when $n$ and $k$ go to infinity with the same rate. We assume that the
size of the box of the diagram goes to zero with the same rate by
introducing the coordinate $u=\frac{a}{n}$.

In \cite{nazarov2024skew} we have demonstrated that random Young
diagrams, rotated by $\frac{\pi}{4}$ and considered as 1-Lipschitz functions,
converge uniformly in probability to the limit curve $\Omega(u)$ that
has an explicit formula for the limit density
$\PP(\sup_{u}|\lambda(u)-\Omega(u)|>\varepsilon )\to 0$ as
$n\to\infty$. In \cite{NNP20} the same was demonstrated for a very
similar problem in $(\SO_{2n+1},\Pin_{2k})$ skew Howe duality. To do
so we have written the measure in the exponential form, used Stirling
approximation for the factorials and solved the variational problem.
But we have not considered the local fluctuations around the limit
shape. Nevertheless in \cite{NNP20} in a similar problem we have
demonstrated that the probability measure can be written in terms of a
semiclassical family of orthogonal polynomials. This representation
was then used to prove Central Limit Theorem for the global
fluctuations. Below we use the polynomials to write the
correlation kernel in the Christoffel-Darboux form and demonstrate its
convergence to the discrete sine kernel in the bulk by analyzing the
asymptotic behavior of the orthogonal polynomials.

In our case the
measure~\eqref{eq:probability-measure-explicit-formula} contains a
Vandermonde determinant written in terms of squares of variables
$a_{i}$. Introduce the variables $y_{i}=\frac{a_{i}^{2}}{n^{2}}$, then measure takes the determinantal form
\begin{equation}
 \label{eq:measure-determinantal}
  \mu_{n,k}(y_{1},\dots,y_{n})=\frac{1}{Z_{n,k}}\prod_{i<j}(y_{i}-y_{j})^{2}\cdot
  \prod_{\ell=1}^{n} y_{\ell} \binom{2n+2k}{n+k-n\sqrt{y_{\ell}}},
\end{equation}
where we have rewritten the weight in the binomial form and collected
the terms that do not depend on $y_{i}$ into the normalization
constant:
 \begin{align}\label{normalization}
 	Z_{n,k}^{-1}=\frac{2^{2n(1-k)}n^{2n^2}}{\prod\limits_{i<j}^n(j-i)(2n+2-i-j)}\prod\limits_{l=1}^n\frac{(2k-1+2l)!}{(2k+2n)!}
 \end{align}
  Denote the weight by $\widetilde{W}(y)$
\begin{equation}
  \label{eq:weight-for-square-lattice}
  \widetilde{W}(y)=y \;\binom{2n+2k}{n+k-n\sqrt{y}},
\end{equation}
and introduce the discrete orthonormal polynomials
$p_{m}(y)$ on a quadratic lattice\\ $\left\{\frac{i^{2}}{n^{2}}| i=0,1,\dots,n+k\right\}$
\begin{equation}
  \label{eq:orthogonality-of-polynomials-on-square-lattice}
  \sum_{i=0}^{n+k} p_{m}\left(\frac{i^{2}}{n^{2}}\right)
  p_{\ell}\left(\frac{i^{2}}{n^{2}}\right)
  \widetilde{W}\left(\frac{i^{2}}{n^{2}}\right) = \delta_{m,\ell}. 
\end{equation}
According to (\ref{eq:weight-for-square-lattice}), the weight $\widetilde{W}(u^2)$
is the binomial coefficient multiplied by $u^2$. The polynomials, that
are orthogonal with the binomial weight are the Krawtchouk polynomials.
The polynomials $p_m$ are related to the Krawtchouk polynomials using
the ``lifting'' procedure from \cite{BUHMANN1992117} which is a
variant of the QR-algorithm for the Christoffel transformation of
orthognal
polynomials~\cite{galant1971implementation},~\cite{bueno2004darboux}
as demonstrated below.

Introduce Christoffel-Darboux correlation kernel as
\begin{equation}
  \label{eq:christoffel-darboux-kernel-on-square-lattice}
\mathcal{K}(y_{i},y_{j})=\sqrt{\widetilde{W}(y_{i}) \widetilde{W}(y_{j})}
\sum_{\ell=0}^{n-1}p_{\ell}(y_{i})\; p_{\ell}(y_{j})  ,
\end{equation}
the measure is then rewritten as its determinant 
\begin{align}
  \notag
    \mu_{n,k}(y_{1},\dots,y_{n})=&\frac{1}{Z_{n,k}}\det\left[\mathcal{K}(y_{i},y_{j})\right]_{i,j=1}^{n}=\\ \label{eq:mu-as-christoffel-darboux}
    =&\frac{1}{Z_{n,k}}\det\left(\sqrt{\widetilde{W}(y_{i}) \widetilde{W}(y_{j})}
      \sum_{l=0}^{n-1}p_{l}(y_{i})\;
      p_{l}(y_{j})\right)_{i,j=1}^{n}.
\end{align}
We are interested in the asymptotic behavior of the probability
measure as $n,k$ go to infinity with the same rate. Therefore we need
to study the asymptotics of the polynomials $\{p_{m}\}$. First note, that
the polynomials $\{p_{m}(y)\}$ with $y=u^{2}$ are just even polynomials
from the set of polynomials $\{g_{m}(u)\}$ that are
orthonormal with respect to the weight $\widetilde{W}(u^{2})$
\begin{equation}
  \label{eq:orthogonality-of-polynomials-on-uniform-lattice}
  \sum_{i=-(n+k)}^{n+k} g_{m}\left(\frac{i}{n}\right)
  g_{\ell}\left(\frac{i}{n}\right)
  \widetilde{W}\left(\frac{i^{2}}{n^{2}}\right) = \delta_{m,\ell}.   
\end{equation}
Since weight $\widetilde{W}(u^{2})$ is even as a function of $u$ and since
each value of $y$ appears in \eqref{eq:orthogonality-of-polynomials-on-uniform-lattice} twice, we have

\begin{equation}
  \label{eq:123}
  p_{m}\left(u^{2}\right)=\sqrt{2}g_{2m}(u), 
\end{equation}
for all m we can write the kernel in terms of the polynomials $\{g_{2m}\}$ using Christoffel-Darboux formula as
\begin{equation}
  \label{eq:correlation-kernel-on-uniform-lattice}
  \mathcal{K}(u,v)=2\frac{\varkappa_{2n-2}}{\varkappa_{2n}}
  \frac{g_{2n}(u)g_{2n-2}(v)-g_{2n-2}(u)g_{2n}(v)}{u^{2}-v^{2}}
  \sqrt{\widetilde{W}(u^{2}) \widetilde{W}(v^{2})},
\end{equation}
where $\varkappa_{m}$ is the coefficient of the $m$-th power
of $u$ in $g_{m}(u)$, with variables $u,v$ living on a
linear lattice
$\left\{-\frac{n+k}{n},-\frac{n+k-1}{n},\dots,\frac{n+k-1}{n},\frac{n+k}{n}\right\}$.

We study the asymptotics $n,k\to\infty$ such that their ratio
tends to constant. Introduce $K=2n+2k$, and $H$ such that $n=HK$, then
we are interested in the limit $K\to\infty$ with $H$ being constant.
Bulk asymptotic regime corresponds to
$u=\frac{1}{H}\left(x+\frac{i'}{n}\right)$,
$v=\frac{1}{H}\left(x+\frac{j'}{n}\right)$ with $i',j'\in\ZZ$. In the
next Section we express the polynomials in correlation kernel in terms
of Krawtchouk polynomials and then study their asymptotic behavior. 

\section{A family of semiclassical orthogonal polynomials}
\label{sec:family-semicl-orth}
\subsection{Symplectic polynomials}~\\
Non-normalised Krawtchouk polynomials $\{K_m\}$ satisfy the orthogonality relation
\begin{align*}
\sum\limits_{a=0}^{K}\binom{K}{a}p^a(1-p)^{K-a}K_m(a,p,K)K_n(a,p,K)=\binom{K}{n}^{-1}\left(\frac{1-p}{p}\right)^n \delta_{nm},
\end{align*}
where $p\in(0,1)$ and we use the following definition of the Pochhammer symbol
\begin{align*}
    (-K)_m=\prod\limits_{i=0}^{m}(-K+i-1)=(-1)^m m!\binom{K}{m}.
\end{align*}
There is a well known representation of the Krawtchouk polynomials as hypergeometric functions~\cite{Koekoek}
\begin{align}\label{krawtchouk-hyp-rep}
    K_{m}(a,p,K)=\,_2F_1\left(-m,-a,-K; \frac{1}{p}\right).
\end{align}
The monic version of the $K_m(a,p,K)$ is defined as (9.11.4 in \cite{Koekoek})
\begin{align}\label{eq:monic-Krawtchouk-rel}
   P_m(a,p,K)=(-K)_m p^m K_{m}(a,p,K).
\end{align}
So the monic Krawtchouk polynomial $P_m(a,p,K)$ now has the following hypergeometric representation (from the (\ref{krawtchouk-hyp-rep}))
\begin{align}\notag
    P_m(a,p,K)&=(-1)^m m!p^m\binom{K}{m} K_m(a,p,K)=\\ \label{monic-krawtchouk-hyp-rep}
    &=(-1)^m m!p^m\binom{K}{m}\,_2F_1\left(-m,-a,-K; \frac{1}{p}\right)
\end{align}
To arrive to the polynomials $g_m(u)$ we first shift the $a$ variable to $a\to K/2+nu$, where the $u$ is the new variable on the symmetric lattice. Denote the monic polynomial with rescaled variable by $\tilde{K}_m(u)\equiv\tilde{K}_m(u,p,K)$. We immediately find their hypergeometric representation
\begin{align*}
    \widetilde{K}_m(u)&=\frac{1}{n^m}P_m(K/2+nu,p,K)=(-1)^m\frac{p^m}{n^m}m!\binom{K}{m}K_m(K/2+nu,p,K)=\\
    &=(-1)^m\frac{p^m}{n^m}m!\binom{K}{m}\,_2F_1\left(-m,-K/2-nu,-K; \frac{1}{p}\right).\numberthis\label{shifted-monic-krawtchouk-hyp-rep}
\end{align*}
 The orthogonality relation for polynomials $\widetilde{K}_m$ is
\begin{align}\label{MSK_orth}
    \sum\limits_{i=-K/2}^{K/2}\widetilde{K}_l(i/n)\widetilde{K}_m(i/n)W(i/n)=n^{-2m}\kappa^{-2}_m\delta_{lm},
\end{align}
where $\kappa_{m}=\frac{(-1)^m}{m!}\binom{K}{m}^{-1/2}(p(1-p))^{-m/2}$,
and we have the relation between the weights $W$ for Krawtchouk polynomials for the variable on the symmetric lattice $u\in \{-\frac{K}{2n},-\frac{K}{2n}+1/n,...,\frac{K}{2n}\}$ and weight $\widetilde{W}$ for $\{g_m(u)\}$,
\begin{align}\label{eq:Weitght_connection}
	u^2 W(u)=u^2\binom{K}{K/2+nu}p^{K/2-nu}(1-p)^{K/2+nu}=\widetilde{W}(u^2).
\end{align}
The
weight $\widetilde{W}(u^2)$, defined by (\ref{eq:weight-for-square-lattice}) differs from the weight
of Krawthouk polynomials $W$ with $K=2n+2k$ only by a factor
$u^{2}$.
\begin{prop}[Christoffel transformation, see \cite{Ismail_2005}, sect. 2.7, Theorem 2.7.1]
	Monic orthogonal polynomials $G_m$
are expressed in terms of the polynomials
\eqref{shifted-monic-krawtchouk-hyp-rep} by the 
formula 
\begin{equation}
  \label{eq:monic-BC-Christoffel-transform}
  G_{m}(u)=\frac{1}{C_{m,2}u^{2}}\left|
  \begin{matrix}
    \widetilde{K}_{m}(0)&     \widetilde{K}_{m+1}(0) &
    \widetilde{K}_{m+2}(0)\\
    \left(\widetilde{K}_{m}\right)'(0)&     \left(\widetilde{K}_{m+1}\right)'(0) &
    \left(\widetilde{K}_{m+2}\right)'(0)\\
    \widetilde{K}_{m}(u)&     \widetilde{K}_{m+1}(u) &
    \widetilde{K}_{m+2}(u)
  \end{matrix}
  \right|,
\end{equation}
where
\begin{equation}
  \label{eq:cn-in-christoffel}
   C_{m,2}=\left|
    \begin{matrix}
      \widetilde{K}_{m}(0)&     \widetilde{K}_{m+1}(0) \\
    \left(\widetilde{K}_{m}\right)'(0)&     \left(\widetilde{K}_{m+1}\right)'(0) 
  \end{matrix}
  \right|.
\end{equation}
\end{prop}
\begin{cor}
	We are interested only in the even polynomials $g_{2m}$.
As monic polynomials $\widetilde{K}_m$ are defined on a symmetric interval, we have
$\widetilde{K}_{2l+1}(0)=0$ and $\widetilde{K}_{2l}'(0)=0$. Then
$C_{2m,2}=\widetilde{K}_{2m}(0)\widetilde{K}_{2m+1}'(0)$ and
\begin{align}
 G_{2m}(u)&=
\frac{1}{u^2}\left(\widetilde{K}_{2m+2}(u)-
\frac{\widetilde{K}_{2m+2}(0)}{\widetilde{K}_{2m}(0)}
\widetilde{K}_{2m}(u)\right).\numberthis \label{eq:orthonormal-BC-Christoffel-transform-explicit}
\end{align}
By the same considerations, $C_{2m+1,2}=-\widetilde{K}_{2m+2}(0)\widetilde{K}_{2m+1}'(0)$ and 
\begin{align}\label{odd_expl}
	G_{2m+1}(u)&=\frac{1}{u^2}\left(\widetilde{K}_{2m+3}(u)-\frac{\widetilde{K}_{2m+3}'(0)}{\widetilde{K}_{2m+1}'(0)}\widetilde{K}_{2m+1}(u) \right).
\end{align} 
\end{cor}

We denote the orthogonality relation as
\begin{align}\label{G_normalized_ort_rel}
    \sum\limits_{i=-K/2}^{K/2}(i/n)^2W(i/n)G_m(i/n)G_k(i/n)=\Lambda_m\delta_{mk},
\end{align}
where $\Lambda_m=||G_m||$ can be computed by the recurrence (see next section). 
The monic symplectic polynomials $\{G_m(u)\}$ have the following hypergeometric realization for even $m$
\begin{align}\notag
		&G_m(u)=\frac{(-1)^m p^{m+2}(m+2)!}{u^2n^{m+2}}\binom{K}{m+2}\bigg(\,_2F_1\left(-m-2,-K/2-nu,-K; \frac{1}{p}\right)-\\\label{eq:symplectic poly hyp rep}
&-\frac{\,_2F_1\left(-m-2,-K/2,-K; \frac{1}{p}\right)}{\,_2F_1\left(-m,-K/2,-K; \frac{1}{p}\right)}\,_2F_1\left(-m,-K/2-nu,-K; \frac{1}{p}\right)\bigg).
\end{align}
It is also useful to note that this hypergeometric representation significantly simplifies for $p=1/2$ because of the Gauss's formula $\frac{\,_2F_1\left(-m-2,-K/2,-K; 2\right)}{\,_2F_1\left(-m,-K/2,-K; 2\right)}= \frac{m+1}{(-K+m+1)}$ (see Appendix),
\begin{align*}
    &G_m(u)= \frac{1}{u^2\kappa_{m+2}n^{m+2}}\binom{K}{m+2}^{1/2}\bigg(\,_2F_1\left(-m-2,-K/2-nu,-K; 2\right)+\\
    &+\frac{(m+1)}{(K-m-1)}\,_2F_1\left(-m,-K/2-nu,-K; 2\right)\bigg).\numberthis \label{G_Hyp_rep}
\end{align*}
\begin{remark}
	We define the symplectic polynomials $\{g_m(u)\}$ as a Christoffel transformation of the polynomials $\{\widetilde{K}_m(u)\}$ (see \eqref{eq:monic-BC-Christoffel-transform}) for any $p\in (0,1)$ from \eqref{shifted-monic-krawtchouk-hyp-rep}. All methods used in this work are defined for $p\in(0,1)$ unless otherwise specified, however, all basic calculations are performed only in the case of $p=1/2$. The reason for this is that symplectic polynomials for $p\neq 1/2$ do not have a clear probabilistic interpretation as \eqref{eq:mu-as-christoffel-darboux}. 
\end{remark}~\\
\subsection{QR algorithm and three term relation's coefficients}~\\
Using the results from \cite{BUHMANN1992117}, one can derive a three-term recurrence relation for symplectic polynomials. This can be achieved by employing the fact that the Christoffel transformation provides a method for constructing orthogonal polynomials with respect to the measure $ x^2 W(x)dx$ from those that are orthogonal with respect to the positive Borel measure $ W(x)dx$.
 
Introducing $\pi_m(x)=n^m\kappa_m\widetilde{K}_m(x)$, one can check that the polynomials $\{\pi_m\}$ satisfy the conditions
\begin{align*}
    &\int\limits_{-\infty}^\infty \pi_m(x)\pi_n(x) \, d\mu(x) = \delta_{mn},\quad\int\limits_{-\infty}^\infty d\mu(x) = 1,  \numberthis \label{eq:orthon-conditions}
\end{align*}  
with the measure
\begin{align*}
    d\mu(x)=dxW(x)\sum\limits_{i=-K/2}^{K/2}\delta(x-i/n),
\end{align*}
 that is automatically normalized,
\begin{align*}
    &\int\limits_{-\infty}^\infty d\mu(x)=\sum\limits_{i=-K/2}^{K/2}W(i/n)=\sum\limits_{i=-K/2}^{K/2}\binom{K}{K/2+i}p^{K/2+i}(1-p)^{K/2-i}=1.
    \end{align*}
Their three term recurrence relation is given by Jacobi matrix
\begin{align}\label{TTR-orthonormal}
   x\pi_m(x)=a_{m+1}\pi_{m+1}(x)+b_m\pi_m(x)+a_m\pi_{m-1}(x),
\end{align} 
with the starting conditions $\pi_{-1}(x)=0,\, \pi_0(x)=1.$   
     It is obvious that the Christoffel transformation from the $\pi_m(x)$ is the same as multiplying the result of Christoffel transformation from $\tilde{K}_m(x)$ by $n^{m+2}\kappa_{m+2}$, so we recall
\begin{align*}
    &x^2 n^{m+2}\kappa_{m+2}G_m(x)=\pi_{m+2}(x)-
\frac{\pi_{m+2}(0)}{\pi_{m}(0)}\pi_{m}(x). \numberthis \label{G_m_from_Pi}
\end{align*}
We also need to compute the second moment
\begin{align*}
    &\Omega_2(K,p,n)=\int\limits_{-\infty}^\infty x^2 d\mu(x)=\sum\limits_{i=-K/2}^{K/2}(i/n)^2\binom{K}{K/2+i}p^{K/2+i}(1-p)^{K/2-i}=\\
    &=\frac{1}{n^2}\sum\limits_{j=0}^K(j-K/2)^2\binom{K}{j}p^{j}(1-p)^{K-j}=\frac{1}{4n^2} K \left(K (1-2 p)^2-4 (p-1) p\right). \numberthis \label{eq:x^2_measure_normalisation}
\end{align*}
 Now we can determine the coefficients of the three term relation for the polynomials $\{\pi_m(x)\}$ from three term relation coefficients $\{\tilde{\alpha}_m,\tilde{\beta}_m\}$ of the monic polynomials $\{\widetilde{K}_m\}$ (see \cite{Koekoek})
\begin{align*}
    \numberthis \label{eq:TTR-Pi-polynomials}
    &x\widetilde{K}_m(x)=\widetilde{K}_{m+1}(x)+\tilde{\alpha}_m\widetilde{K}_m(x)+\tilde{\beta}_m\widetilde{K}_{m-1}(x)=\widetilde{K}_{m+1}(x)+\\
    &+[p(K-m)+m(1-p)-K/2]\widetilde{K}_m(x)+\frac{m(K-m+1)}{n^2}p(1-p)\widetilde{K}_{m-1}(x).
\end{align*}
 So we get the Jacobi matrix coefficients for (\ref{TTR-orthonormal})
\begin{align} \label{am_TTR}
   & a_m=\frac{\kappa_{m-1}}{n\kappa_m}=-\frac{\big[m(K-m+1)(p(1-p))\big]^{1/2}}{n}=-\sqrt{\tilde{\beta}_m},\\ \label{bm_TTR}
   &b_m=\left[p(K-m)+m(1-p)-K/2\right]=\tilde{\alpha}_m.
\end{align}
For $p=1/2$ the weight function $W(x)$ becomes symmetric and $b_m=0$ as expected. 

Following the Theorem 3.3 in \cite{BUHMANN1992117}, consider the Jacobi matrix of the three term relation coefficient for the orthonormal polynomials with the measure $\mu$. Denote by $\omega$ the natural mapping from the set of all normalized, real, positive Borel measures $\mathcal{B}$ to the set of all singly-infinite, symmetric, tridiagonal irreducible matrices $\mathcal{J}$
\begin{align*}
	\omega:\mathcal{B}\to \mathcal{J},
\end{align*}
that takes every member of $\mathcal{B}$ to the corresponding Jacobi matrix. By the Favard theorem \cite{chihara2011introduction} this mapping is invertible if and only if the underlying Hamburger problem is determinate \cite{akhiezer1965classical}. By the inverse mapping we mean $\omega^{-1}:\mathcal{J}\to \mathcal{B}^*$, where $\mathcal{B}^*$ is the quotient by the equivalence classes of the measures with identical moment sequences. In our case the measure $d\mu (x)$ defined above has the discrete and finite support, which means that the moment problem for $\text{Im}(\omega)$ is always determinate and the equivalence class of the element $\omega^{-1}(J)$ consists of a single element (see \cite{Simon97}, \cite{Christiansen2004IndeterminateMP}, \cite{Simon99}).
  Then the QR algorithm gives the reversible arrows in the diagram 
\begin{equation*}
 \begin{tikzcd}
 &J \arrow{rr}{\tau}
 &&\widehat{J}\\
 & d\mu \arrow{u}{\omega} \arrow{rr}[swap]{\rho} 
 && d\tilde{\mu} \arrow{u}[swap]{\omega}
\end{tikzcd}   
\end{equation*}
where the map $\tau$ is a step of $QR$ algorithm, as applied to a square matrix $J$, replaces it by the product
$\widehat{J} = RQ$, where $J = QR$, Q is orthogonal and R is upper triangular. 

 We can identify the Cristoffel transformation of $\pi_m(x)$ with the normalized measure transformation, that is defined by the commutative diagram above
\begin{align*}
    \omega^{-1}\circ \tau \circ \omega =\rho:d\mu(x)\to d\tilde{\mu}(x)=\frac{x^2 d\mu(x)}{\Omega_2(K,p,n)}.
\end{align*}
$QR$ factorization of Jacobi matrix $J$ is its decomposition into the product $J=Q_1^TQ_2^T...Q_{n-1}^TR$, where $Q_i$ are block-diagonal antisymmetric matrices, which turn the subdiagonal elements to zero (Givens rotation). For example, letting $r_0=\sqrt{a_1^2+b_0^2}$, we take $Q_1$ to be identity except that the upper left $2\times2$ block is
\begin{align*}
	Q_1\big|_{( e_1,e_2) }=\begin{bmatrix}
		 b_0/r_0 & a_1/r_0\\
		 -a_1/r_0 & b_1/r_0
	\end{bmatrix}.
\end{align*}
Then the result of the first Givens rotation acting on the initial matrix $J$ is
\begin{align*}
Q_1 J \;=\;
\begin{bmatrix}
  r_0 & (a_1b_0 + a_1b_1)/r_0 & (a_1a_2)/r_0 &  \\
      &(b_0b_1 - a_1^2)/r_0 & (b_0a_2)/r_0 & 0 \\
  &a_2 & b_2 & a_3 &    \ddots   \\
      && a_3 & b_3 & \ddots  \\
      & &    & \ddots & \ddots 
\end{bmatrix}
=
\begin{bmatrix}
  r_0 & s_1         & (a_1a_2)/r_0 &        \\
  &b_1^*      & a_2^*              &   0     \\
  &a_2 & b_2        & a_3            &     \ddots   \\
  &    & a_3        & b_3                & \ddots \\
      &  &            & \ddots             & \ddots
\end{bmatrix}.
\end{align*}
The trailing coefficients \( b_i^* \) and \( a_{i+1}^* \), obtained by applying \( i \) successive Givens rotations to \( J \), serve as the parameters for the subsequent rotation \( Q_{i+1} \). The nontrivial part of $Q_{i+1}$, in its $(k+1)$st and $(k+2)$nd rows and columns, has the form
\begin{align*}
	Q_{i+1}\big|_{( e_{i+1},e_{i+2}) }=\begin{bmatrix}
		 b_i^*/r_i & a_{i+1}/r_i\\
		 -a_{i+1}/r_i & b_i^*/r_i
	\end{bmatrix}.
\end{align*}
 The resulting upper triangular matrix $R$ is defined by diagonal elements $r_0, r_1,\dots$ and superdiagonal elements $s_1, s_2,\dots$, respectively, where
\begin{align*}
   &a^*_1=a_1,\quad b_0^*=b_0,\\
   &a_k^*=\frac{a_kb^*_{k-2}}{r_{k-2}}, \quad  b^*_k=\frac{b^*_{k-1}b_k-a^*_ka_k}{r_{k-1}},\\
   & r^2_k=a^2_{k+1}+b^{*2}_k, \quad s_k=\frac{a_k^*b^*_{k-1}+a_kb_k}{r_{k-1}}=\frac{a_kb_{k-2}^*b_{k-1}^*}{r_{k-2}r_{k-1}}+\frac{a_{k}b_k}{r_{k-1}}.
\end{align*}
The elements above the superdiagonal in the upper-triangular matrix \( R \) are not involved in the QR-transformed matrix \( \widehat{J} \) and are irrelevant for our purposes.  
The elements of the QR-transformed Jacobi matrix \( \widehat{J} = RQ \) are defined by 
\begin{align*}
    & \hat{a}_k=\frac{a_kr_k}{r_{k-1}},\quad \hat{b}_k=\frac{b_{k-1}^*b_k^*}{r_{k-1}}+\frac{a_{k+1}s_{k+1}}{r_k},\numberthis\label{CD-trans-TTR_0}
\end{align*}
with $\hat{b}_0=b_0+\frac{a_1s_1}{r_0}$.

The three-term relation for modified polynomials is then
\begin{align}\label{hat_g_TTR}
    x\hat{g}_m(x)=\hat{a}_{m+1}\hat{g}_{m+1}(x)+\hat{b}_m\hat{g}_m(x)+\hat{a}_m\hat{g}_{m-1}(x),
\end{align}
where $\hat{g}_m(x)$ satisfy orthogonal relation
\begin{align}\label{symp_orthon_rel}
	\sum\limits_{i=-K/2}^{K/2}\hat{g}_m(i/n)\hat{g}_l(i/n)(i/n)^2\frac{W(i/n)}{\Omega_2(K,p,n)}=\delta_{ml}.
\end{align}
There is the direct computation of the recurense relations of the normalization $\Lambda_m$, which appears in \eqref{G_normalized_ort_rel} as the norm of the symplectic polynomials. Let us define the three term relation 
\begin{align}\label{TTR_monic sympl}
	xG_m(x)=G_{m+1}(x)+\alpha_m G_m(x)+\beta_m G_{m-1}(x),
\end{align}
then 
\begin{align}
	\sum_{x\in \mathfrak{X}}G_m(x)G_l(x)x^2W(x)=\delta_{ml}\Lambda_m, \quad \mathfrak{X}=\bigg\{-\frac{K}{2n},-\frac{K}{2n}+\frac{1}{n},...,\frac{K}{2n}\bigg \}.
\end{align}
Denoting $n^{-2m}\kappa_m^{-2}=L_m$ in \eqref{MSK_orth}, $S_{2l}=-\widetilde{K}_{2l+2}(0)/\widetilde{K}_{2l}(0)$ in \eqref{eq:orthonormal-BC-Christoffel-transform-explicit}, $S_{2l+1}=-\widetilde{K}_{2l+3}'(0)/\widetilde{K}_{2l+1}'(0)$ in \eqref{odd_expl}, one can compute the unknown square of norm $\Lambda_m=||G_m||^2$ and coefficients of three-term relations from the following nonlinear recursions.
\begin{prop}
	The squared norms and the coefficients of the three-term recurrence relation for the monic symplectic polynomials are governed by a nonlinear recurrence relation of the form below
	\begin{align} \label{rec1}
	&S_mL_m=\beta_m \Lambda_{m-1},\\ \label{rec2}
&L_{m+2}+S_m^2L_{m}=\Lambda_{m+1}+\beta_{m}^2 \Lambda_{m-1}+\alpha_m^2 \Lambda_m,\\ \label{rec3}
	&0=\beta_{m}\alpha_{m-1}\Lambda_{m-1}+\alpha_m \Lambda_m.
\end{align}
\end{prop}
\begin{proof}
From the definition of the Christoffel transformation it follows that
\begin{align}\notag
	&\sum_{x\in \mathfrak{X}}G_m(x)G_l(x)x^4W(x)=\sum_{x\in \mathfrak{X}}(\widetilde{K}_{m+2}(x)+S_m \widetilde{K}_m(x))(\widetilde{K}_{l+2}(x)+S_l \widetilde{K}_l(x))W(x)=\\ 
	&=\delta_{m,l}L_{m+2}+\delta_{m,l}S_m^2L_m+\delta_{m,l+2}S_{m}L_m+\delta_{m+2,l}S_lL_l.
\end{align}
On the other hand, using \eqref{TTR_monic sympl}, we get
\begin{align}\notag
	&\sum_{x\in \mathfrak{X}}G_m(x)G_l(x)x^4W(x)=\\ \notag
	&=\sum_{x\in \mathfrak{X}}(G_{m+1}(x)+\alpha_{m}G_m(x)+\beta_m G_{m-1}(x))(G_{l+1}(x)+\alpha_{l}G_l(x)+\beta_l G_{l-1}(x))x^2W(x)=\\ \notag
	&=\delta_{m,l}\Lambda_{m+1}+\delta_{m+2,l}\beta_l \Lambda_{m+1}+\delta_{m,l+2}\beta_m \Lambda_{m-1}+\delta_{m,l}\beta_m^2 \Lambda_{m-1}+\delta_{ml}\alpha_m^2\Lambda_m+\\
	&+\delta_{m+1,l}\alpha_l\Lambda_{m+1}+\delta_{m-1,l}\beta_m \alpha_l \Lambda_{m-1}+\delta_{m,l+1}\alpha_m\Lambda_{m}+\delta_{m,l-1}\alpha_m\beta_{l}\Lambda_m.
\end{align}
Comparing these two results we get the recurrences
\begin{align}
	(m=l+2)&:\quad S_mL_m=\beta_m \Lambda_{m-1},\\ 
	(m=l)&:\quad L_{m+2}+S_m^2L_{m}=\Lambda_{m+1}+\beta_{m}^2 \Lambda_{m-1}+\alpha_m^2 \Lambda_m,\\
	(m=l+1)&:\quad 0=\beta_{m}\alpha_{m-1}\Lambda_{m-1}+\alpha_m \Lambda_m.
\end{align}
\end{proof}
Now, comparing \eqref{symp_orthon_rel}
 with \eqref{G_normalized_ort_rel}, we get
 $$g_m(x)=\frac{\hat{g}_m(x)}{\sqrt{\Omega_2}},$$
then
\begin{align}
	\tilde{a}_m^2=\hat{a}_m^2=\beta_m,\quad \tilde{b}_m=\hat{b}_m=\alpha_m.
\end{align}
\begin{lemma}
	The coefficients of three term relation os symplectic polynomials for $p=1/2$ can be computed from \eqref{am_TTR} and \eqref{bm_TTR} as
	 \begin{align} \label{b_hat_dir}
    &\hat{b}_k=0 \; (k\geq 0),\\ \label{a_hat_dir}
    & \hat{a}^2_{2m}=\frac{a^2_{2m}a^2_{2m+1}}{r^2_{2m-1}},\quad \hat{a}^2_{2m+1}=r^2_{2m+1},
\end{align}
where 
\begin{align}
	r_{2m+1}^2=a_{2m+2}^2+a_{2m+1}^2\prod\limits_{j=1}^{m}\frac{a_{2j-1}^2}{r_{2j-1}^2}.
\end{align}
\end{lemma}
\begin{proof}
It is not so hard to see that in the case $p=1/2$ all $b_m=0$ and the Givens rotatation's parameters become
\begin{align*}
   &a^*_1=a_1=\frac{(-1)}{2HK^{1/2}}=\frac{x}{\pi_1(x)},\quad b_0^*=b_0=0=b_m,\\
   &a_k^*=\frac{a_kb^*_{k-2}}{r_{k-2}}=\frac{-a_ka_{k-2}^*a_{k-2}}{r_{k-2}r_{k-3}}, \quad  b^*_k=\frac{-a^*_ka_k}{r_{k-1}}.
   \end{align*}
 And now we can express diagonal and superdiagonal elements of the matrix $R$ as 
 \begin{align*}
   & r^2_k=a^2_{k+1}+\frac{(a^*_ka_k)^2}{r_{k-1}^2}, \quad s_k=\frac{a_k^*b^*_{k-1}}{r_{k-1}}=\frac{-a_k^*a_{k-1}^*a_{k-1}}{r_{k-1}r_{k-2}},\\
   &r_0^2=a_1^2, \quad r_1^2=a_2^2+\frac{a_1^4}{r_0^2}=a_2^2+a_1^2, \quad r_2^2=a_3^2,\\
   &a_2^*=0 ,\quad a_3^*=-\frac{a_3a_1^2}{r_0r_1}=-\frac{a_3r_0}{r_1}, \quad a_4^*=0.
\end{align*}
Here $a_0$ should be the constant on which nothing depends. Then we can compute
\begin{align*}
    &a_{2m}^*=0,\quad a_{2m+1}^*=-\frac{a_{2m+1}a_{2m-1}^*a_{2m-1}}{r_{2m-1}r_{2m-2}}=-\frac{a_{2m+1}a_{2m-1}^*}{r_{2m-1}},\\
    &b_{2m}^*=0, \quad b_{2m+1}^*=-a^*_{2m+1}, \quad s_k=0, \; k>0.
    \end{align*}
The recursion on the diagonal coefficients of $R$ gives
    \begin{align} \label{r_dir1}
    &r_{2m}^2=a_{2m+1}^2,\\ \notag
    &r_{2m+1}^2=a_{2m+2}^2+\frac{(a_{2m+1}^*a_{2m+1})^2}{r_{2m}^2}=a_{2m+2}^2+a_{2m+1}^{*2}=\\ \notag
    &=a_{2m+2}^2+\frac{(a_{2m+1}a_{2m-1}a_{2m-3}^*)^2}{r^2_{2m-1}r_{2m-3}^2}=\\ \label{r_dir2}
    &=a_{2m+2}^2+a_{2m+1}^2\prod\limits_{j=1}^{m}\frac{a_{2j-1}^2}{r_{2j-1}^2}=a_{2m+2}^2+\frac{\prod\limits_{j=1}^{m+1}a_{2j-1}^2}{\sum\limits_{\ell=-1}^{m-1}\left(\prod\limits_{k=0}^{\ell}a_{2m-2k}^2\right)\left(\prod\limits_{j=1}^{m-\ell-1}a_{2j-1}^2\right)}.
 \end{align}
Here we have expressed the product directly, assuming that $\prod\limits_{k=0}^{-1} (..)=1$,
\begin{align}\notag
    &\prod\limits_{j=1}^m r^2_{2j-1}=\left(a_{2m}^2+a_{2m-1}^2\prod\limits_{j=1}^{m-1}\frac{a_{2j-1}^2}{r_{2j-1}^2}\right)\prod\limits_{j=1}^{m-1} r^2_{2j-1}=\prod\limits_{j=1}^{m}a_{2j-1}^2+a_{2m}^2\prod\limits_{j=1}^{m-1}r_{2j-1}^2=\\ \notag
    &=\prod\limits_{j=1}^{m}a_{2j-1}^2+a_{2m}^2\prod\limits_{j=1}^{m-1}a_{2j-1}^2
    +...+\left(\prod_{k=0}^{\ell}a_{2m-2k}^2\right)\left(\prod_{j=1}^{m-\ell-1}a_{2j-1}^2\right)+...+\\\label{Euler_ft}
    &+\prod_{k=0}^{m-1}a_{2m-2k}^2=\sum\limits_{\ell=-1}^{m-1}\left(\prod_{k=0}^{\ell}a_{2m-2k}^2\right)\left(\prod_{j=1}^{m-\ell-1}a_{2j-1}^2\right).
\end{align}
Finally we express the QR transformed coefficients as
\begin{align}
    &\hat{b}_0=0, \quad \hat{b}_k=0 \; (k>0),\\ 
    & \hat{a}^2_{2m}=\frac{a^2_{2m}a^2_{2m+1}}{r^2_{2m-1}},\quad \hat{a}^2_{2m+1}=r^2_{2m+1},
\end{align}
which ends the proof of the lemma.
\end{proof}

The expression for $r_{2m+1}^2$, more suitable for asymptotic analysis is given by the next lemma.
\begin{lemma}\label{lem:r2m+1}
Let $K,m,n\in \mathbb{Z}_{\geq 0}$.  Suppose the sequence \(\{r_{2m+1}^2\}\) satisfies the recurrence
\begin{align*}
r_{2m+1}^2=a_{2m+2}^2+a_{2m+1}^2\prod\limits_{j=1}^{m}\frac{a_{2j-1}^2}{r_{2j-1}^2}.
\end{align*}
Then $r_{2m+1}^2$ can be expressed as
\begin{align}
	\frac{r_{2m+1}^2}{a_{2m+2}^2}=\frac{\mathcal{R}^{[2m+1]}_{m}}{\mathcal{S}^{[2m+1]}_{m}},
\end{align}
where $\mathcal{R}_k^{[2m+1]},\,\mathcal{S}_k^{[2m+1]}$ satisfy three term reccurenses
\begin{align}\label{WKB1}
	&\mathcal{R}_k^{[2m+1]}=B_k^{[2m+1]}\mathcal{R}^{[2m+1]}_{k-1}+A_k^{[2m+1]} \mathcal{R}_{k-2}^{[2m+1]},\\ \label{WKB2}
	& \mathcal{S}_k^{[2m+1]}=B_k^{[2m+1]}\mathcal{S}_{k-1}^{[2m+1]}+A_k^{[2m+1]} \mathcal{S}_{k-2}^{[2m+1]},
\end{align}
with starting conditions
\begin{align*}
	&\mathcal{R}_{-1}^{[2m+1]}=1,\quad \mathcal{S}^{[2m+1]}_{-1}=0,\\
	& \mathcal{R}_0^{[2m+1]}=B_0^{[2m+1]}, \quad \mathcal{S}_0^{[2m+1]}=1,
\end{align*}
and the coefficients that are rational functions of $2m+1$
\begin{align*}
	&A_k^{[2m+1]}=-\frac{a_{2m+1-2(k-1)}^2}{a_{2m+2-2(k-1)}^2}, \quad B_{k}^{[2m+1]}=1+\frac{a_{2m+1-2k}^2}{a_{2m+2-2k}^2}=1-A_{k+1}.
\end{align*}
\end{lemma}

\begin{proof}
Using the recursion
\begin{align*}
	&r_{2m+1}^2=a_{2m+2}^2+a_{2m+1}^2\left(\prod\limits_{j=1}^{m-1}\frac{a_{2j-1}^2}{r_{2j-1}^2}\right)\frac{a_{2m-1}^2}{r_{2m-1}^2}=a_{2m+2}^2+\frac{a_{2m+1}^2a_{2m-1}^2}{r_{2m-1}^2}\frac{(r_{2m-1}^2-a_{2m}^2)}{a_{2m-1}^2}=\\
	&=a_{2m+2}^2+a_{2m+1}^2\bigg(1-\frac{a_{2m}^2}{r_{2m-1}^2}\bigg),
\end{align*}
we get
\begin{align*}
\frac{r_{2m+1}^2}{a_{2m+2}^2}=1+\frac{a_{2m+1}^2}{a_{2m+2}^2}\left(1-\frac{1}{r_{2m-1}^2/a_{2m}^2}\right), \quad \frac{r^2_{1}}{a_{2}^2}=1+\frac{a_1^2}{a_2^2}.
\end{align*}
This can be expressed as a convergent of the continued fractions\footnote{In this sense, formula \eqref{Euler_ft} is analogous to the Euler's continued fraction formula $$\gamma_0 + \gamma_0 \gamma_1 + \gamma_0 \gamma_1 \gamma_2 + \cdots + \gamma_0 \gamma_1 \gamma_2 \cdots \gamma_n =
\frac{\gamma_0}{1 +} \, \frac{-\gamma_1}{1 + \gamma_1 +} \, \cfrac{-\gamma_2}{1 + \gamma_2 +} \cdots \frac{-\gamma_n}{1 + \gamma_n}.$$}
\begin{align*}
	\frac{r_{2m+1}^2}{a_{2m+2}^2}&=\bigg(1+\frac{a_{2m+1}^2}{a_{2m+2}^2}\bigg)-\frac{a_{2m+1}^2/a_{2m+2}^2}{r_{2m-1}^2/a_{2m}^2}=\\
	&=\bigg(1+\frac{a_{2m+1}^2}{a_{2m+2}^2}\bigg)-\frac{a_{2m+1}^2/a_{2m+2}^2}{\left(1+\dfrac{a_{2m-1}^2}{a_{2m}^2}\right)-\dfrac{a_{2m-1}^2/a_{2m}^2}{r_{2m-3}^2/a_{2m-2}^2}}=\\
	&=B_0^{[2m+1]}+\frac{A_1^{[2m+1]}}{B_1^{[2m+1]}+}\;\frac{A_2^{[2m+1]}}{B_2^{[2m+1]}+}...\frac{A_{m}^{[2m+1]}}{B_{m}^{[2m+1]}},
\end{align*}
where
\begin{align*}
	&A_k^{[2m+1]}=-\frac{a_{2m+1-2(k-1)}^2}{a_{2m+2-2(k-1)}^2}, \quad B_{k}^{[2m+1]}=1+\frac{a_{2m+1-2k}^2}{a_{2m+2-2k}^2}=1-A_{k+1}^{[2m+1]}.
	\end{align*}
Then there exist the unique solutions of three term reccurence for $k\geq 1$ (see, for example, \cite{deiftorthogonal}) 
\begin{align*}
	&\mathcal{R}_k^{[2m+1]}=B_k^{[2m+1]}\mathcal{R}^{[2m+1]}_{k-1}+A_k^{[2m+1]} \mathcal{R}_{k-2}^{[2m+1]}, \\
	&\mathcal{S}_k^{[2m+1]}=B_k^{[2m+1]}\mathcal{S}_{k-1}^{[2m+1]}+A_k^{[2m+1]} \mathcal{S}_{k-2}^{[2m+1]},
\end{align*}
with starting conditions
\begin{align*}
	&\mathcal{R}_{-1}^{[2m+1]}=1,\quad \mathcal{S}^{[2m+1]}_{-1}=0,\\
	& \mathcal{R}_0^{[2m+1]}=B_0^{[2m+1]}, \quad \mathcal{S}_0^{[2m+1]}=1,
\end{align*}
such that 
\begin{align*}
	\frac{r_{2m+1}^2}{a_{2m+2}^2}=\frac{\mathcal{R}^{[2m+1]}_{m}}{\mathcal{S}^{[2m+1]}_{m}}.
\end{align*}
\end{proof}
\begin{cor}
Now we note that in the double-scaling limit $n=H K,\;m=\frac{\lambda}{2}K,\; \lambda\in (0,1),\; H\in (0,1/2),\; K\to \infty$ one can find for $ k+1\leq m$
\begin{align}
	\mathcal{S}_{k}^{[2m+1]}=\mathcal{S}^{[2m+1]}_{k-1}-A_{k+1}^{[2m+1]}\mathcal{S}_{k-1}^{[2m+1]}+A_{k}^{[2m+1]}\mathcal{S}_{k-2}^{[2m+1]}=\mathcal{S}_{k-2}^{[2m+1]}+O(1/K).
\end{align}
For odd m, we extend the recursion \eqref{WKB1},\eqref{WKB2} with the conditions
\begin{align*}
	&A_{k}^{[2m+1]}=\begin{cases}-\frac{a_{2m+1-2(k-1)}^2}{a_{2m+2-2(k-1)}^2},\quad k\leq m,\\
	0,\quad k>m,
\end{cases}\\
&B_{k}^{[2m+1]}=\begin{cases}
	1-A_{k+1},\quad k<m,\\
	1+a_1^2/a_2^2,\quad k=m,\\
	1,\quad k>m
\end{cases}
\end{align*}
so that the three-term relations stabilize at $k\geq m+1$ and we can write
\begin{align}
	\frac{r_{2m+1}^2}{a_{2m+2}^2}=\frac{\mathcal{R}^{[2m+1]}_{m}}{\mathcal{S}^{[2m+1]}_{m}}=\frac{\mathcal{R}^{[2m+1]}_{m+1}}{\mathcal{S}^{[2m+1]}_{m+1}}=\frac{\mathcal{R}^{[2m+1]}_{0}}{\mathcal{S}^{[2m+1]}_{0}}+O(1/K)=1+\frac{a_{2m+1}^2}{a_{2m+2}^2}+O(1/K).
\end{align}
Then, the asymptotics of the coefficients of three term relation \eqref{hat_g_TTR} for $p=1/2$ are
\begin{align}
	    &\hat{b}_0=0, \quad \hat{b}_k=0 \; (k>0),\\
    & \hat{a}^2_{2m}=\frac{a^2_{2m}a^2_{2m+1}}{r^2_{2m-1}}=\frac{a_{2m}^2a_{2m+1}^2}{a_{2m}^2+a_{2m-1}^2}+O(1/K),\\
    &\hat{a}^2_{2m+1}=r^2_{2m+1}=a_{2m+1}^2+a_{2m+2}^2+O(1/K).
\end{align}

In order to find the next terms of the asymptotic expansion in the limit under consideration, it is useful to note that the coefficients of the three-term recurrence relations \eqref{WKB1} and \eqref{WKB2} are slowly varying functions\footnote{A function \( f(k) \) is referred to as slowly varying if \( f(k) = F(\eta k) \), where \( \eta \) is a small parameter and \( F \) is a smooth function.} of the natural parameter:  
\begin{align}  
	&A_{k+1}^{[2m+1]}-A_{k}^{[2m+1]}=-\left(\frac{a_{2m+1-2k}^2}{a_{2m+2-2k}^2}-\frac{a_{2m+1-2(k-1)}^2}{a_{2m+2-2(k-1)}^2}\right)= O(1/K^2).  
\end{align}  
One may then employ the WKB method \cite{Braun1978TMF} to construct the asymptotic solution of the three-term recurrence relations \eqref{WKB1} and \eqref{WKB2}. 
\end{cor}
\begin{remark}
An alternative nonlinear recursion for determining the coefficients $\tilde{a},\tilde{b}$ when $p\in(0,1)$ follows from equations \eqref{rec1}, \eqref{rec2}, and \eqref{rec3}.  Eliminating $\Lambda_m$ from these equations, we obtain
\begin{align}
L_{m+2} + S_{m}^2 L_m
&= \frac{S_{m+2} L_{m+2}}{\beta_{m+2}} + \beta_m S_m L_m+ \alpha_m^2 \frac{S_{m+1} L_{m+1}}{\beta_{m+1}},\\
S_m L_m\alpha_{m-1}
&= -\alpha_m \frac{S_{m+1} L_{m+1}}{\beta_{m+1}}.
\end{align}
The first of these equations is nonlinear and of second order.  It remains an open question whether this system can be reduced to two coupled nonlinear first‐order recurrences, for example by methods analogous to those in \cite{Dzhamay_2018} and \cite{Dzhamay_2020}.

In the special case $p=1/2$, one finds $\alpha_m=0$, and the remaining recurrence
\begin{align}\label{rec4}
L_{m+2}+S_m^2L_m=\frac{S_{m+2} L_{m+2}}{\beta_{m+2}}+\beta_m S_mL_m=\Lambda_{m+1}+\frac{S_m^2L_m^2}{\Lambda_{m-1}},
\end{align}
becomes first order on the subsequences $\beta_{2m}$ and $\beta_{2m+1}$, recovering directly the analogue of the recursion from Lemma 1.
\end{remark}

\begin{table}[h]
    \caption{Monic Krawtchouk polynomials and monic symplectic polynomials for $p=1/2$.}
  \label{tab:polynomials}
  \centering
    \begin{tabular}{|l|l|}
      \hline
      $\widetilde{K}_{m}(x)$ & $G_{m}(x)$ \\
      \hline
      $1$ & $1$\\
      $x$ & $x$\\
      $x^{2} - \frac{K}{4 \, n^{2}}$ & $x^{2} - \frac{3 \, K+2}{4 \, n^{2}}$\\
      $x^{3} - \frac{1}{4} \, x {\left(\frac{3 \, K-2}{n^{2}}\right)}$ & $x^3-\frac{1}{4}x\frac{\left(15 K^2-30 K+16\right) }{(3 K-2) n^2}$\\
      $x^{4} - \frac{1}{2} \, x^{2} {\left(\frac{3 \, K-4}{n^{2}} \right)} + \frac{3 \, K^{2}-6\,K}{16 \, n^{4}}$& $x^4-\frac{5}{2}x^{2} \frac{(K-2)}{n^2}+\frac{15 K^2-50
   K+24}{16 n^4}$\\
      $x^{5} - \frac{5}{2} \, x^{3} {\left(\frac{K-2}{n^{2}}\right)} + \frac{1}{16} \, x {\left(\frac{15 \, K^{2}-50 \, K+24}{n^{4}}\right)}$& $x^5-\frac{5}{2}x^{3}\frac{ 21 K^3-126 K^2+224 K-96}{ \left(15 K^2-50 K+24\right)
                                                                                                                                        n^2}+$\\
      \quad & \quad $+x\frac{525 K^4-4200 K^3+11340 K^2-11840 K+4416}{16 \left(15 K^2-50 K+24\right) n^4}$\\
      $x^{6} - \frac{5}{4} \, x^{4} {\left(\frac{3 \, K-8}{n^{2}}\right)} + \frac{1}{16} \, x^{2} {\left(\frac{45 \, K^{2}-210 \, K+184}{n^{4}} \right)} -$ & $x^6+x^{4}\frac{70-21 K}{4 n^2}+\frac{7}{16}x^{2} \frac{15 K^2-90 K+112}{n^4}-$\\
      \quad $-\frac{15 \, K^{3}+90  \, K^{2}-120\, K}{64 \, n^{6}}$ & \quad $-\frac{3}{64}\frac{35 K^3-280 K^2+588 K-240}{n^6}$\\
      \hline
\end{tabular}
\end{table}
For illustrative purposes we present first several monic Krawtchouk polynomials $\{\widetilde{K}_{m}(x)\}$ and monic symplectic polynomials $\{G_{m}(x)\}$ for $p=1/2$ in Table ~\ref{tab:polynomials}.

\section{Asymptotics of the polynomials}
\label{sec:asympt-polyn}
In investigating the asymptotic properties of the Christoffel--Darboux kernel \eqref{eq:correlation-kernel-on-uniform-lattice}, we are led, via the analysis of equation \eqref{eq:orthonormal-BC-Christoffel-transform-explicit}, to the necessity of determining the asymptotic behavior of the polynomials $\{\widetilde{K}_m(j/n)\}$.

In this section we analyse the double-scaling limit of the Krawtchouk polynomials. Specifically, we examine the limit in which the density of nodes tends to infinity while the lattice size increases simultaneously. This corresponds to the asymptotic regime given by:
\[
n = H K, \quad m = g K, \quad j\in [-K/2,K/2], \quad K \to \infty,
\]  

where $ g$ is fixed parameter, \( 0 < g < 1 \), and the critical point is located at \( g = \frac{1}{2} \).  

Numerous studies have explored the asymptotic behavior of the Krawtchouk polynomials $\widetilde{K}_m(j/n)$ using methods such as steepest descent and the Riemann-Hilbert approach (see \cite{ISMAIL1998121}, \cite{LI2000155}, \cite{Wang_Qui_2015}). In our analysis, we aim to extend these computations to the regime where the lattice variable \( j \) remains unbounded in the double-scaling limit.  

To ensure that \( j \) is not constrained, we introduce a new variable \( \mu \) defined as  
\[
j = \frac{K}{2} (1 - 2\mu), \quad j\in [-K/2,K/2],
\]  

where \( 0 < \mu < 1 \) and \( \mu \neq \frac{1}{2} \). Now let us consider on the asymptotic expansion of
\begin{equation}
  \label{eq:monic-rescaled-krawtchouk-integral_i}
  \widetilde{K}_{m}(j/n)=(-1)^m
  \frac{m!}{n^{m}}\int_{\Gamma} (1-(1-p)z)^{K/2-j}
  (1+pz)^{K/2+j} z^{-m} 
  \frac{\dz}{2\pi i z} ,
\end{equation}
where the contour $\Gamma$ is a positive oriented closed unity loop that encloses 0, $|z|=1$. Changing the variable $t=1/(1-(1-p)z)$ one can find the integral
\begin{align}\notag
&I=\int_{\Gamma} (1-(1-p)z)^{K/2-j}
  (1+pz)^{K/2+j} z^{-m} 
  \frac{\dz}{2\pi i z} =\\ \label{Kr_int_lemma}
  &=-(-1)^{m+1}\frac{(1-p)^{m-j-K/2}}{2\pi i}\int_{\Gamma'}\frac{dt}{t(1-t)}\exp\left\{K\ln{\frac{(t-p)^{1-\mu}}{(1-t)^g}} t^{g-1}\right\},
\end{align}
where the contour $\Gamma'$ is a smooth negative -oriented circle around 1.  Now we can use steepest descent to analyze the asymptotic behavior of the integral.
\begin{prop}[see \cite{fedoryuk1977}]\label{SDP_Prop} The asymptotic behavior of the contour integral of the form $\int\limits_{\Gamma}G(t)e^{Kf(t)}\dt$ in the limit $K\to \infty$ is given by the next series
\begin{align*} \numberthis \label{SDPeq}
	&\int\limits_{\Gamma'}G(t)e^{Kf(t)}dt\bigg|_{K\to \infty}=\sum\limits_{\ell}e^{K f(t_0^\ell)}\left(G(t_0^\ell)\sqrt{\frac{2\pi}{K |f''(t_0^\ell)|}}e^{i\theta_{\text{SDP}}^\ell}+O(1/K^{3/2})\right),\\
	& \theta_{\text{SDP}}^\ell=-\frac{\alpha_\ell}{2}\pm \frac{\pi}{2},\quad \alpha_\ell=\arg (f''(t_0^\ell)),
\end{align*}
where the $\theta_{\text{SDP}}^\ell=\arg\left(\frac{dt}{d\rho}\right)_{\rho=0}$ is the tangent angle of the steepest descent path in the saddle point $t_0^\ell$, $\rho$ is the parameter along the minimax contour and the sign in the expression of $\theta_{\text{SDP}}^\ell$ defined by the steepest descent path orientation in the saddle points.
\end{prop}
In our case we have
\begin{align*}
	&G(t)=\frac{1}{t(1-t)},\\
	&f(t)=(1-\mu)\ln(t-p)+(g-1)\ln t-g\ln (1-t).
\end{align*}
The equation $f'(t)=0$ has two solutions, which we will index by $\pm$,
\begin{align*}
	&t_0^{\pm}(p)=\frac{1}{2 \mu }\left(\pm\sqrt{(-g+\mu +p)^2-4 \mu  p(1-g)}-g+\mu +p\right).\end{align*}
Limiting our consideration with the case $p=1/2$, we can simplify this expression to
\begin{align*}
	&t_0^{\pm}(1/2)\equiv t_0^\pm=\frac{\sqrt{2\mu(1-g)}}{2\mu}e^{\pm i\chi},
\end{align*}
where for $(\mu-1/2)^2\leq g(1-g)$ \footnote{Here and below we use the next definition of Heaviside $\Theta$-function
\begin{align*}
	\Theta(x)=\begin{cases}
		0,\quad x\leq 0,\\
		1,\quad x> 0.
	\end{cases}
\end{align*}
},
\begin{align*}
	\chi =\tan^{-1}\left(\frac{\sqrt{(1-g) g-\left(\mu -1/2\right)^2}}{\mu+1/2-g}\right)+\pi\Theta(g-1/2-\mu),
\end{align*}
and in the case $(\mu-1/2)^2> g(1-g)$ we have two real roots
\begin{align*}
	 t_0^\pm=\frac{1}{2\mu}\left(\pm\sqrt{(\mu-1/2)^2-g(1-g)}+(\mu+1/2-g)\right),
\end{align*}
In fact it is easy to see that in general
\begin{align*}
	t_0^\pm=\frac{1}{2\mu}\bigg(\pm\sqrt{(g+\mu -1/2)^2-2 g \mu}-((g-\mu)-1/2)\bigg).
\end{align*}
So exactly up to factor $\frac{1}{2\mu}$ and changing $\mu\to p$, roots of the action in the considered limit the same as the roots of the generalized Krawtchouk polynomials with $p\neq 1/2$, described above. 
Starting from the region of two complex-conjugate roots $1/2-\sqrt{\mu(1-\mu)}<g<1/2+\sqrt{\mu(1-\mu)}$ for a fixed $\mu$ (for a fixed $g$, this region is defined by the inequality $1/2-\sqrt{g(1-g)}<\mu<1/2+\sqrt{g(1-g)}$) \footnote{This is also a direct consequence of the self-duality of the Krawtchouk polynomials.}, let us choose the parametrization $g=1/2+d\sqrt{\mu(1-\mu)}, \;d\in[-1,1]$. Then one can easily define two branches of a rational plane curve $t_0^\pm(d;\mu)=x(d;\mu)+iy_{\pm}(d;\mu)$,
\begin{align*}
	t_0^{\pm}(d;\mu)=\frac{1}{2\mu}\bigg(\mu-d\sqrt{\mu(1-\mu)}\pm i \sqrt{	\mu(1-\mu)(1-d^2)}\bigg).
\end{align*}
It is not so hard to see that $t_0(d;\mu)$ for fixed $\mu$ form a circle with the equation
\begin{align*}
	\left(x-\frac{1}{2}\right)^2+y_{\pm}^2=\frac{(1-\mu)}{4\mu}.
\end{align*}
When we take the limit $\mu \to 0$ we have the circle with infinite radius. For $\mu\to 1$ this circle disappeares symmetrically at $(1/2,0)$ by hyperbollic scaling. There is critical point $(1,0)$ when $\mu =1/2,\, g=0$ and the unique real root goes to center of the integration. The points of the intersection of this circle and line $y=0$ are
\begin{align*}
	x=\frac{1}{2}\pm \frac{1}{2}\sqrt{\frac{1-\mu}{\mu}}.
\end{align*}
In the case of two real roots we have two real sections instead of a circle
\begin{align*}
	(g-1/2)^2>\mu(1-\mu) \; \Longleftrightarrow \; g(1-g)<(\mu-1/2)^2,
\end{align*}
so this case is realized for the parametric inequality
\begin{align*}
	&g>1/2+\sqrt{\mu(\mu-1)},\quad g<1/2-\sqrt{\mu(1-\mu)}.
\end{align*}
Using the parametrization
\begin{align}\notag
	&g(1-g)=\tilde{d}(\mu-1/2)^2,\quad 0<\tilde{d}<1,\\ \label{ineq1}
	&g=1/2\pm \sqrt{1/4-\tilde{d}(\mu-1/2)^2},
\end{align}
where the sign is consistent with the sign in the inequalities above, we can immediately write
\begin{align*}
	&t_0^+(-,\tilde{d})=\frac{1}{2\mu}\bigg((\mu-1/2)\sqrt{1-\tilde{d}}+\mu+\sqrt{1/4-\tilde{d}(\mu-1/2)^2}\bigg),\\
	&t_0^-(-,\tilde{d})=\frac{1}{2\mu}\bigg(-(\mu-1/2)\sqrt{1-\tilde{d}}+\mu+\sqrt{1/4-\tilde{d}(\mu-1/2)^2}\bigg),
\end{align*}
for minus sign, and
\begin{align*}
	&t_0^+(+,\tilde{d})=\frac{1}{2\mu}\bigg((\mu-1/2)\sqrt{1-\tilde{d}}+\mu-\sqrt{1/4-\tilde{d}(\mu-1/2)^2}\bigg),\\
	&t_0^-(+,\tilde{d})=\frac{1}{2\mu}\bigg(-(\mu-1/2)\sqrt{1-\tilde{d}}+\mu-\sqrt{1/4-\tilde{d}(\mu-1/2)^2}\bigg),
\end{align*}
for the plus sign, where the signs under the brackets denote the choice of the sign in \eqref{ineq1}. These real sections touch each other in pairs at the points $x=\frac{1}{2}\pm \frac{1}{2}\sqrt{\frac{1-\mu}{\mu}} $, when $\tilde{d}=1$.

Let us focus on the case of two complex conjugate roots. The asymptotic behavior of the polynomials $\widetilde{K}_m(j/n)$ is given by the following lemma.
\begin{lemma}
	In the double scaling regime defined by 
	\begin{align}
	&n = H K, \quad m = g K, \quad j = \frac{K}{2} (1 - 2\mu), \quad K\to \infty,
	\end{align}
	where $\mu, g\in (0,1)\backslash\{1/2\}$, and the parameters $\mu, g $ satisfy the following relation (bulk)
	\begin{align}
		(\mu-1/2)^2\leq g(1-g),
	\end{align} 
	the asymptotic of the polynomials $\widetilde{K}_m(i/n)$ is defined by
	\begin{align}\notag 
	&\widetilde{K}_{m}(j/n)=\left(\sin(\theta_{\text{SDP}}^--\hat{\delta}(j;m))+O(1/K)\right)\cdot\\ \label{K_asympt_unb}
	&\cdot\frac{m!}{n^{m}}\sqrt{\frac{2}{\pi K}}\frac{(\sqrt{1-g})^{m-K-1/2}(\sqrt{1-\mu})^{K/2+j+1/2}}{(\sqrt{g})^{m+1/2}(\sqrt{\mu})^{j-K/2-1/2}}
	\frac{ (1/2)^{m-K/2}}{\left((1-g) g-\left(\mu -\frac{1}{2}\right)^2\right)^{1/4}}.
\end{align}
The functions $\theta_{\text{SDP}}^-$ and $\hat{\delta}(j;m)$ in the expression above is defined by
\begin{align}
&\theta_{\text{SDP}}^-=-\frac{\arg(f''(t^{-}_0))}{2}+\frac{\pi}{2},\quad\hat{\delta}(j;m)=K \Im (f(t_0^+))+\arg G(t_0^+),
\end{align}
where 
\begin{align}
	&G(t)=\frac{1}{t(1-t)},\quad f(t)=(1-\mu)\ln(t-1/2)+(g-1)\ln t-g\ln (1-t),\\
	&t_0^+=\frac{1}{2\mu}\bigg(\sqrt{(g+\mu -1/2)^2-2 g \mu}-((g-\mu)-1/2)\bigg).
\end{align}
\end{lemma}
\begin{proof}
For proving the lemma we will use the Proposition \ref{SDP_Prop} for the integral \eqref{Kr_int_lemma}. One can express all main values from the equation (\ref{SDPeq})
\begin{align}
&\Re(f(t_0^\pm))=\frac{1}{2} \left(\ln\left(\frac{1-\mu }{2 (1-g)}\right)+2 g \tanh ^{-1}(1-2 g)+\mu  \ln\left(\frac{4
   \mu }{1-\mu}\right)\right),\\
   & \Im (f(t_0^\pm))=\pm \bigg((1-\mu)\gamma+(g-1)\chi -g\tau\bigg)=\pm \bigg(\zeta_1+g\zeta_2-\mu\gamma\bigg),\\
   &G(t_0^\pm)=\frac{1}{t_0^\pm(1-t_0^\pm)}=\frac{2\mu}{g}\sqrt{\frac{g}{1-g}}e^{\mp i\big[\zeta_2+2\tau\big]},\\
   & \arg f''(t_0^\pm)= \pm \bigg[
    \tan^{-1}\left(\frac{(g-1/2)\left((g-1/2)^2 - \frac{3}{4}\mu + \frac{\mu^2}{2}\right)}{\sqrt{g(1-g) - (\mu-1/2)^2} \cdot \left(\frac{\mu}{4} - (g-1/2)^2\right)}\right) + \sigma 
    +\nonumber\\
    &\quad+ \pi \cdot \Theta\left(\frac{1}{2} - \frac{\sqrt{\mu}}{2} - g\right) 
    - \pi \cdot \Theta\left(g - \left(\frac{1}{2} + \frac{\sqrt{\mu}}{2}\right)\right)
\bigg],
\end{align}
in the compact form, through the functions $\tau, \gamma, \zeta_1, \zeta_2, \sigma$, which are defined as follows. Functions $\tau$ and $\gamma$ are naturally defined as in the case of function $\chi$ as the arguments of the shifted roots of the function $f$
\begin{align*}
	 1-t_0^\pm=\sqrt{\frac{g}{2\mu}}e^{\pm i \tau}, \quad \tau=  	-\tan ^{-1}\left(\frac{\sqrt{(1-g) g-\left(\mu -1/2\right)^2}}{g+\mu-1/2}\right)-\pi \Theta(1/2-g-\mu),&\\
	t_0^\pm-1/2=\frac{\sqrt{\mu(1-\mu)}}{2\mu}e^{\pm i \gamma},\quad \gamma=  	\tan^{-1}\left(\frac{\sqrt{(1-g) g-\left(\mu -1/2\right)^2}}{1/2-g}\right)+\pi \Theta(g-1/2).&
\end{align*}
Functions $\zeta_1$ and $\zeta_2$ are the angles of these shifts relative to $\chi$,
\begin{align*}
	&\zeta_1=\gamma-\chi=	\tan^{-1}\left(\frac{\sqrt{g(1-g)-(\mu-1/2)^2}}{3/2-g-\mu}\right)+\pi\Theta(\mu+g-3/2),\\
	&\zeta_2=\chi -\tau=
	\tan^{-1}\left(\frac{\sqrt{g(1-g)-(\mu-1/2)^2}}{\mu-1/2}\right)+\pi\Theta(1/2-\mu).
\end{align*}
The $\sigma$ is defined from the denominator of $-\pi\leq \arg f''(t_0^\pm)\leq \pi$ as
\begin{align*}
	\sigma = -2\pi \cdot \Theta\left(g - \frac{1}{2} + \frac{\sqrt{(2\mu - 3)\mu(1-\mu)}}{2\sqrt{\mu - 1}}\right) + 2\pi \cdot \Theta\left(g - \frac{1}{2} - \frac{\sqrt{(2\mu - 3)\mu(1-\mu)}}{2\sqrt{\mu - 1}}\right),
\end{align*}
Now we need to define a Steepest Descent Path (SDP). First of all, lets define the lines of constant phase
\begin{align}\notag
	& \Im[f(t)-f(t_0^\pm)]=\arg \left(\frac{(t-1/2)^{1-\mu}}{(t_0^\pm-1/2)^{1-\mu}}\frac{t^{g-1}}{(t_0^\pm)^{g-1}}\frac{(1-t_0^\pm)^g}{(1-t)^g}\right)=\\
	\label{SDP_Im_mu}
	&=(1-\mu)\arg (t-1/2)+(g-1)\arg (t)-g\arg(1-t)\mp \bigg(\zeta_1+g\zeta_2-\mu\gamma\bigg)=0,
\end{align}
for $t\in \mathbb{C},\,\mu,g\in \mathbb{R},\,\mu\in (0,1),\, g\in (1/2-\sqrt{\mu(1-\mu)},1/2+\sqrt{\mu(1-\mu}))$, and
\begin{align*}
	&\Re[f(t)-f(t_0^\pm)]\leq 0,\\
	&\Re[f(t_0^\pm)]=\frac{1}{2} \left(\ln\left(\frac{1-\mu }{2 (1-g)}\right)+2 g \tanh ^{-1}(1-2 g)+\mu  \ln\left(\frac{4
   \mu }{1-\mu}\right)\right).
\end{align*}
The solutions of these equations are symmetric under complex conjugation \(t \mapsto \overline{t}\), as a consequence of the logarithmic behavior of the action \(f(t)\); their structure is illustrated in Figure~\ref{Conts}.
\begin{figure}[htbp]
  \centering
  \begin{subfigure}[b]{0.499\textwidth}
    \centering \includegraphics[width=\textwidth]{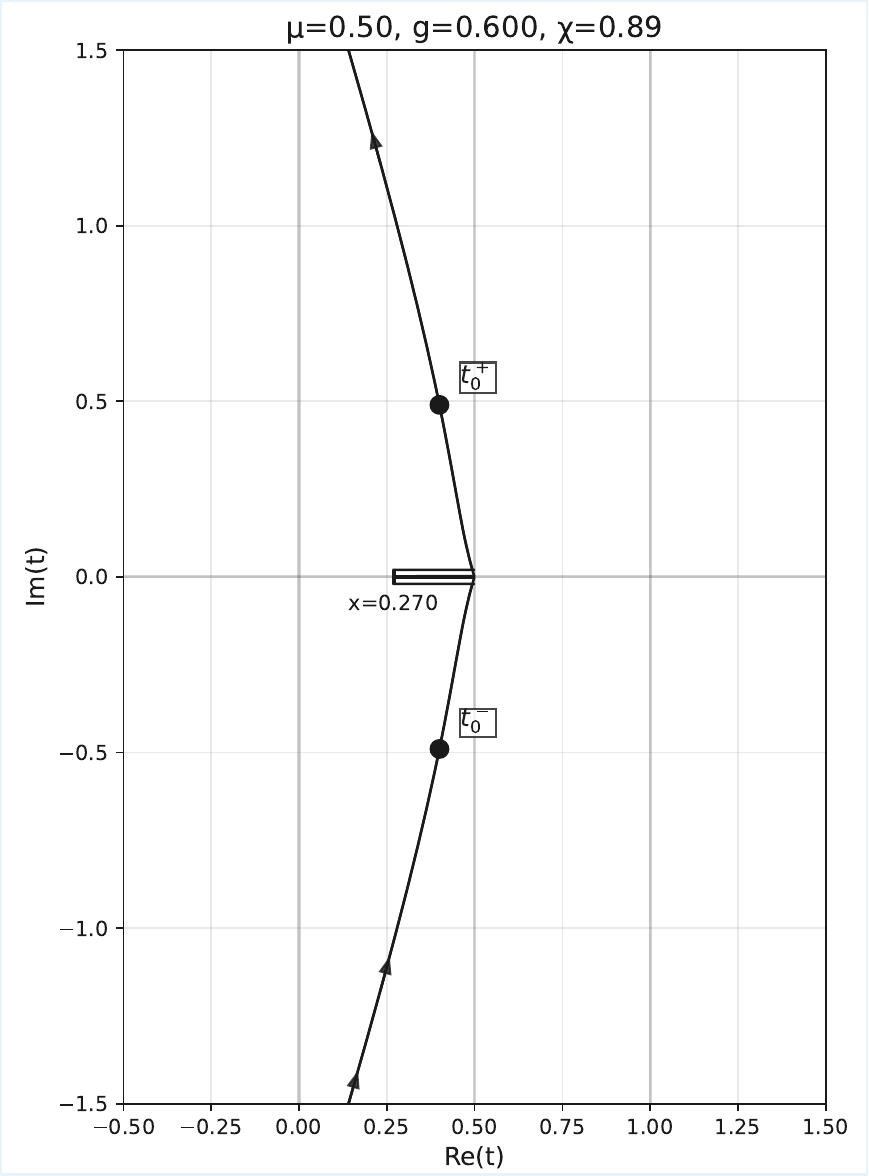}
    \caption{}
  \end{subfigure}\hspace{0.002\textwidth}%
  \begin{subfigure}[b]{0.49\textwidth}
    \centering
\includegraphics[width=\textwidth]{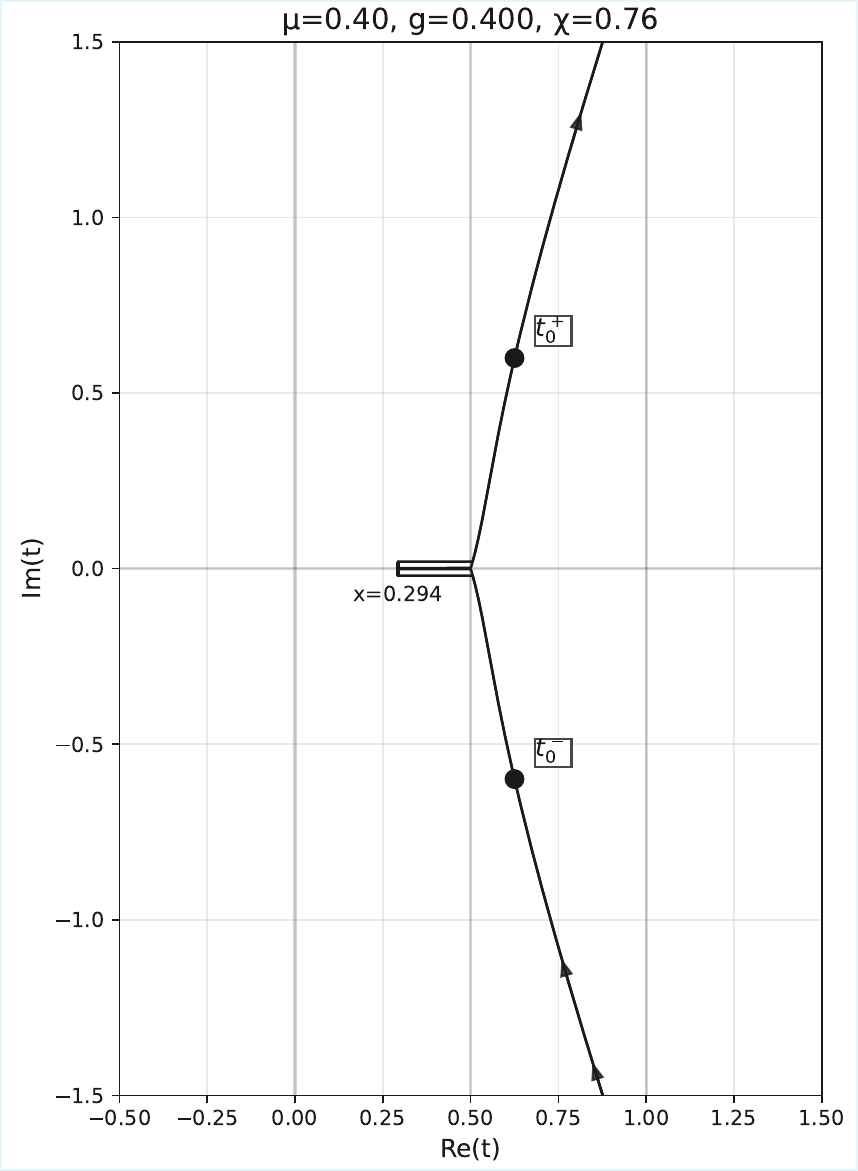}
\caption{}
  \end{subfigure}
  \caption{Steepest-descent paths in the complex plane for representative values of the parameters \(\mu\) and \(g\). Filled dots indicate the stationary (saddle) points \(t_0^\pm\) of the action \(f(t)\). Arrows show the orientation of contour traversal used in the integrals. Rectangles mark two horizontal contour segments, that should be  intersecting the real axis (\(\Im t=0\)) at the point \(x\).}
  \label{Conts}
\end{figure}
By the symmetry of the SDP branches
\begin{align*}
	\theta^+_{\text{SDP}}=-\theta_{\text{SDP}}^-+\pi,
\end{align*}
where
\begin{align}
	\theta_{\text{SDP}}^-=-\frac{\arg(f''(t^{-}_0))}{2}+\frac{\pi}{2},
\end{align}
and the sign of the second part is based on the following consideration.
Fix $\mu$ to some arbitrary value from its domain and denote the values $L:=1/2-\sqrt{(\mu-1)\mu},\,R:=1/2+\sqrt{(\mu-1)\mu}$. From the SDP's described above, in the case $L<g<1/2$ we have
\begin{align*}
	&\lim_{g\to 1/2-}\theta_{SDP}^-=\pi/2<\theta_{SDP}^-<\lim_{g\to L+}\theta_{SDP}^-=\pi, \\ 
	&\lim_{g\to 1/2-}\alpha_-=\frac{1}{2}\lim_{g\to 1/2-}\arg(f''(t_0^-))=\arg (\mu^2)=0.
\end{align*}
Then in this region of $g$ it follows
\begin{align*}
	\theta_{\text{SDP}}^-=- \frac{\arg(f''(t_0^-))}{2}+\frac{\pi}{2}, \quad L<g<1/2.
\end{align*}
In the case $1/2<g<R$ we have
\begin{align*}
	\lim_{g\to R-}\theta_{SDP}^-=0<\theta_{SDP}^-<\lim_{g\to 1/2+}\theta_{SDP}^-=\pi/2,
\end{align*}
so
\begin{align*}
	&\theta_{\text{SDP}}^-=- \frac{\arg(f''(t_0^-))}{2}+\frac{\pi}{2},\quad 1/2<g<R\\
	&\theta_{\text{SDP}}^-=\pi/2,\quad g=1/2.
\end{align*}
 
We can now construct the asymptotic of the initial integral
\begin{align*}
	&I=(-1)^{m}\frac{(1-p)^{m-j-K/2}}{2\pi i}\int_{\Gamma'}g(t)e^{Kf(t)}dt\bigg|_{K\sim \infty}=\\
  &= (-1)^{m}(1/2)^{m-K/2}\frac{(\sqrt{1-g})^{m-1/2}}{(\sqrt{g})^{m+1/2}}\left(\sqrt{\frac{\mu }{1-\mu }}\right)^{K/2-j}\left( 
   \sqrt{\frac{1-\mu}{1-g}}\right)^K \cdot\\
   &\cdot \sqrt{\frac{2}{\pi K}}\left(\frac{ (1-\mu ) \mu }{(1-g) g-\left(\mu -\frac{1}{2}\right)^2}\right)^{1/4} \left(\sin(\theta_{\text{SDP}}^--\hat{\delta}(j;m))+O(1/K)\right),
  \end{align*}
  where:
  \begin{align*}
  	\hat{\delta}(j;m)=(m-1)\zeta_2+K\zeta_1+(j-K/2)\gamma-2\tau.
  \end{align*}
  Using the fact that $K$ is even we can continue the simplification of $\hat{\delta}$ by defining new variables
 	\begin{align*}
 		&\hat{\zeta}_1=\tan^{-1}\left(\frac{\sqrt{g(1-g)-(\mu-1/2)^2}}{3/2-g-\mu}\right),\\
 		& \hat{\tau}=\tan ^{-1}\left(\frac{\sqrt{(1-g) g-\left(\mu -1/2\right)^2}}{g+\mu-1/2}\right),
 	\end{align*}
 	in which the definition of $\hat{\delta}$ depends only on the location of $\mu,g$ relative to $1/2$ value
 	\begin{align}\label{delta_mu_def}
 		\hat{\delta}(j;m)=(m-1)\zeta_2+K\hat{\zeta_1}-K\mu\gamma+2\hat{\tau},
 	\end{align}
 	and after further simplification of the integral asymptotic we have
 	\begin{align*}
 	&I=(-1)^{m}\frac{(1-p)^{m-j-K/2}}{2\pi i}\int_{\Gamma'}G(t)e^{Kf(t)}dt\bigg|_{K\sim \infty}=\left(\sin(\theta_{\text{SDP}}^--\hat{\delta}(j;m))+O(1/K)\right)\cdot\\
 		&= \frac{(\sqrt{1-g})^{m-K-1/2}(\sqrt{1-\mu})^{K/2+j+1/2}}{(\sqrt{g})^{m+1/2}(\sqrt{\mu})^{j-K/2-1/2}}\sqrt{\frac{2}{\pi K}}\frac{ (-1)^{m}(1/2)^{m-K/2}}{\left((1-g) g-\left(\mu -\frac{1}{2}\right)^2\right)^{1/4}}.
 	\end{align*}
Substituting this expression to \eqref{eq:monic-rescaled-krawtchouk-integral_i}, we end the proof of the lemma.
\end{proof}
\begin{cor}
Using the equation (\ref{eq:orthonormal-BC-Christoffel-transform-explicit}) we can find the asymptotic behavior for $G_m(i/n)$ in the double scaling regime $m=gK$, $n=HK, \;K\to \infty$ for even $m$,
\begin{align}\notag
	&G_m(i/n)=\frac{1}{i^2}n^2\left(\widetilde{K}_{m+2}(i/n)-
\frac{\widetilde{K}_{m+2}(0)}{\widetilde{K}_{m}(0)}\widetilde{K}_{m}(i/n)\right)=\\ \notag
&=\frac{(m+2)!}{n^m i^2}\sqrt{\frac{2}{\pi K}}\frac{(\sqrt{1-g})^{m-K+3/2}(\sqrt{1-\mu})^{K/2+i+1/2}}{(\sqrt{g})^{m+5/2}(\sqrt{\mu})^{i-K/2-1/2}}\frac{ (1/2)^{m+2-K/2}}{\left((1-g) g-\left(\mu -\frac{1}{2}\right)^2\right)^{1/4} }\cdot\\ \label{G_m_Asymptotic_unb}
&\cdot \bigg(\sin \big(\theta_{SDP}^--\hat{\delta}(i;m+2)\big)+\frac{(K-m)}{m+2}\frac{g}{1-g}\sin\big(\theta_{SDP}^--\hat{\delta}(i;m)\big)+O(1/K)\bigg).
\end{align}
\end{cor}
\section{Limiting behavior of the correlation kernel}
  \label{sec:limit-behav-corr}
The main goal of this section is to give a proof of the Theorem \ref{thm:sine-limit}. To do that, we compute the asymptotic of the correlation kernel (\ref{eq:correlation-kernel-on-uniform-lattice}) in following regime of the symplectic polynomials
\begin{align}
	m=gK,\, K=2n+2k,\, n=HK,
\end{align}
with $K\to \infty$ and fixed $g,H$, and $m$ is the number of polynomial.
\begin{proof}[Proof of Theorem \ref{thm:sine-limit}.]
The equation (\ref{G_m_Asymptotic_unb}) for $m=2n$ has the following form
\begin{align}\notag
	&G_{2n}(i/n)=g_{2n}(i/n)\sqrt{\Lambda_{2n}}=\frac{1}{(i/n)^2}\frac{(2n+2)!}{n^{2n+2}}F_{2n}(i)\cdot \\ \label{sin_g_m_asymp}
	&\bigg(\sin \big(\theta_{SDP}^--\hat{\delta}(i;2n+2)\big)+\sin\big(\theta_{SDP}^--\hat{\delta}(i;2n)\big)+O(1/K)\bigg),
\end{align}
where $F_{2n}(i)$ is defined as
\begin{align*}
	&F_{2n}(i)=\frac{(\sqrt{1-2H})^{2n-K+3/2}(\sqrt{1-\mu})^{K/2+i+1/2}}{(\sqrt{2H})^{2n+5/2}(\sqrt{\mu})^{i-K/2-1/2}}\sqrt{\frac{2}{\pi K}}\frac{ (1/2)^{2n+2-K/2}}{\left((1-2H) 2H-\left(\mu -\frac{1}{2}\right)^2\right)^{1/4} },
\end{align*}
with $g=2H$ and in general $0<H<1/2$ with $H=1/4$ for $n=k$ (square Young tableaux).
Now we substitute asymptotic of the polynomials into the kernel to demonstrate its convergence to the discrete sine-kernel
\begin{align*}	
  &\mathcal{K}(u,v)=2\frac{\varkappa_{2n-2}}{\varkappa_{2n}}
  \frac{g_{2n}(u)g_{2n-2}(v)-g_{2n-2}(u)g_{2n}(v)}{u^{2}-v^{2}}
  \sqrt{\widetilde{W}(u^{2}) \widetilde{W}(v^{2})}=\\
  &=\frac{(1/2)^{K-1}}{\Lambda_{2n-2}}\frac{G_{2n}(u)G_{2n-2}(v)-G_{2n-2}(u)G_{2n}(v)}{u^2-v^2}uv\sqrt{\binom{K}{K/2+nu}\binom{K}{K/2+nv}}.
\end{align*}
In fact, substituting $nu=i,vn=j$ and \eqref{rec1} we have
\begin{align}\notag
	&\frac{(1/2)^{K-1}}{\Lambda_{2n-2}}\frac{G_{2n}(u)G_{2n-2}(v)-G_{2n-2}(u)G_{2n}(v)}{u^2-v^2}uv\sqrt{\binom{K}{K/2+nu}\binom{K}{K/2+nv}}= \\ \notag
	&=\frac{(1/2)^{K-1}\beta_{2n-1}}{S_{2n-1}L_{2n-1}}\frac{F_{2n}(i)F_{2n-2}(j)}{(u^2-v^2)uv}\frac{(2n+2)!(2n!)}{n^{4n+2}}\sqrt{\binom{K}{K/2+i}\binom{K}{K/2+j}}\cdot \\ \notag
	&\cdot\Bigg(\bigg(\sin \big(\theta^-_{\text{SDP}}(i;2n+2)-\hat{\delta}(i;2n+2)\big)+\sin\big(\theta^-_{\text{SDP}}(i;2n)-\hat{\delta}(i;2n)\big)\bigg)\cdot \\ \label{CD_trig}
	&\cdot \bigg(\sin \big(\theta^-_{\text{SDP}}(j;2n)-\hat{\delta}(j;2n)\big)+\sin\big(\theta^-_{\text{SDP}}(j;2n-2)-\hat{\delta}(j;2n-2)\big)\bigg)-[ij]\bigg)+O(1/K)\Bigg),
\end{align}
where $[ij]$ is the last summand with the transposition $i\leftrightarrow j,$ and we use the fact
\begin{align*}
	\frac{F_{2n-2}(i)F_{2n}(j)}{F_{2n}(i)F_{2n-2}(j)}=1.
\end{align*}
Now using the notation
\begin{align*}
	\lambda(i;2n)=\theta^-_{\text{SDP}}(i;2n)-\hat{\delta}(i;2n).
\end{align*}
we aim to derive the asymptotic of $\lambda$ in the bulk regime for $-1/2<x<1/2$
\begin{align*}
&u=\frac{1}{H}\left(x+\frac{i'}{K}\right),\quad v=\frac{1}{H}\left(x+\frac{j'}{K}\right),\\
	&i=un=xK+i'=\frac{K}{2}\bigg(1-2\mu(i')\bigg),\quad j=vn=x K+j',
\end{align*}
with $\mu(i')=-\frac{1}{2}\left(2x+2i'/K-1\right)$.
Let us determine the domain of the pieswise functions~\\ $\gamma,\hat{\zeta}_1,\hat{\zeta}_2,\hat{\tau}, \arg f''(t_0^-))$. 
We need to choose the interval between the critical points of the action. So we can expand the piecewise functions for example in  $g>1/2,\, \mu>1/2$ ($-K/2<i<0$)
\begin{align*}
	&\gamma=\tan^{-1}\left(\frac{\sqrt{(1-g) g-\left(\mu -\frac{1}{2}\right)^2}}{1/2-g}\right)+\pi,\quad\zeta_2=\tan ^{-1}\left(\frac{\sqrt{(1-g) g-\left(\mu -\frac{1}{2}\right)^2}}{\mu -\frac{1}{2}}\right),\\
	&\hat{\zeta}_1=\tan^{-1}\left(\frac{\sqrt{g(1-g)-(\mu-1/2)^2}}{3/2-g-\mu}\right),\quad \hat{\tau}=\tan ^{-1}\left(\frac{\sqrt{(1-g) g-\left(\mu -\frac{1}{2}\right)^2}}{g+\mu-1/2}\right),\\
 		&\arg f''(t_0^-)=\tan ^{-1}\left(\frac{(2 g-1) (2 \mu -1)}{\sqrt{-4 (g-1) g-(1-2 \mu )^2}}\right)+\\
 		&+2 \tan
   ^{-1}\left(\frac{-\frac{\mu }{g-1/2}+2 g-1}{\sqrt{-4 (g-1) g-(1-2 \mu )^2}}\right)+\pi.
\end{align*}
In the other domains of parameter space such as $g>1/2,\mu<1/2$ we will get the same formulas up to $\pm \pi$. These contributions will cancel each other out in final answer. 
Next we derive the asymptotic of trigonometric part of the kernel accurate up to $O(1/K)$. It is obvious that the shift $m\to m+2$ generate the shift $g\to g+2/K$, so, by the definition of $\lambda(i;2n)$
\begin{align*}
\lambda(i;m+2)=\lambda(i;m)-2\zeta_2-2\frac{\partial \zeta_2}{\partial g}g-\frac{\partial \hat{\zeta}_1}{\partial g}2+2\mu\frac{\partial \gamma}{\partial g}+O(1/K).
    \end{align*} 
    Denoting:
    \begin{align*}
    	&\beta=\left(-2\zeta_2-\frac{\partial \hat{\zeta}_1}{\partial g}2+\frac{\partial \gamma}{\partial g}\right)=2 \cot ^{-1}\left(\frac{x}{\sqrt{2H-4H^2-x^2}}\right)+O(1/K),\\
    	&\gamma_i=\lambda(i;2n),
    \end{align*}
  we can reduce all trigonometric part of the CD kernel by the simple calculation to:
 \begin{align*}
 	&\bigg(\sin \lambda(i;2n+2)+\sin\lambda(i;2n)\bigg)\cdot \bigg(\sin \lambda(j;2n)+\sin\lambda(j;2n-2)\bigg)-[ij]=\\
 	&=\bigg(\sin \gamma_i+\sin\gamma_i\cos\beta+\cos\gamma_i\sin\beta\bigg)\bigg(\sin \gamma_j+\sin\gamma_j\cos\beta-\cos\gamma_j\sin\beta\bigg)-\\&-[ij]+O(1/K)= -2\sin \beta (1+\cos \beta)\sin(\gamma_i-\gamma_j),
 \end{align*}
 where:
 \begin{align*}
 	&-2\sin \beta (1+\cos \beta)=\frac{4 x^3 \sqrt{-g^2+g-x^2}}{(g-1)^2 g^2}\bigg(1+O\left(\frac{1}{K}\right)\bigg),\\
 	&\gamma_i-\gamma_j=-(i'-j')\left(\pi+\tan^{-1}\left(\frac{2\sqrt{2H-4H^2-x^2}}{1-4H}\right)\right)\bigg(1+O\left(\frac{1}{K}\right)\bigg).
 \end{align*}

The normalization of the kernel is not important because it can be cancelled out by the renormalization of the polynomials. Therefore we have established the pointwise convergence of the Christoffel--Darboux kernel to the discrete sine-kernel
\begin{align} \label{Sine_Ans}
	\lim\limits_{n\to \infty}\frac{\mathcal{K}\left(\frac{1}{H}\big(x+\frac{i'}{K}\big),\frac{1}{H}\big(x+\frac{j'}{K}\big)\right)}{\mathcal{K}(x/H,x/H)}=\frac{\sin (\pi \rho(x)(i'-j'))}{\pi \rho(x)(i'-j')},
\end{align}
with the limit density
\begin{align} \label{Dens_Ans}
	\rho(x)=\frac{1}{\pi}\cos^{-1}\left(\frac{1-4H}{\sqrt{1-4x^2}}\right),
\end{align}
which ends the proof of Theorem \ref{thm:sine-limit}.
\end{proof}

The limit density coincides with the result for the limit shape of symplectic diagrams obtained in \cite[Thm. 5.1]{nazarov2024skew} after substituting $c=\lim\frac{k}{n}=\frac{1-2H}{2H}$.

In Figure \ref{fig:Sin1} we present the comparison of the correlation kernel from sampling multiple diagrams by Proctor's algorithm to the discrete sine kernel. We see good agreement of the simulation results with the asymptotic expression even for the relatively small values of $n=50, k=100$. 

\begin{figure}
  \begin{center}
  \includegraphics[width=\linewidth]{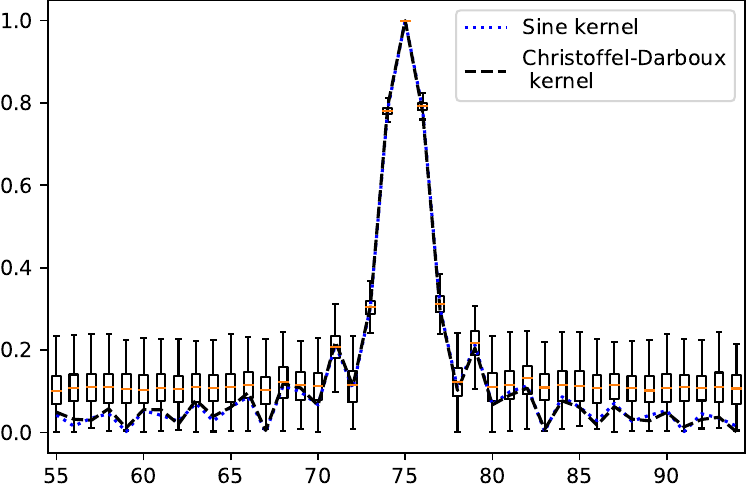}
  \end{center}
  \caption{Plot of the ratio $\frac{\mathcal{K}(i,j)}{\mathcal{K}(i,i)}$ for $i=75$ obtained by sampling $500$ samples, each of $10000$ random Young diagrams inside of $n\times k$ box with $n=50, k=100$ using Proctor's algorithm ({\it box and whiskers plot}) and its comparison to the same ratio for the Christoffel-Darboux kernel \eqref{eq:christoffel-darboux-kernel-on-square-lattice} ({\it dashed black line}) and discrete sine kernel ({\it dotted blue}) \eqref{Sine_Ans}, \eqref{Dens_Ans}
  }  \label{fig:Sin1}
\end{figure}
\section{Conclusion and future work}
In this paper, we have investigated the local fluctuations of random Young diagrams arising from the skew Howe duality for symplectic groups $\Sp_{2n} \times \Sp_{2k}$. The probability measure on diagrams, proportional to the ratio of the dimension of the irreducible $\Sp_{2n} \times \Sp_{2k}$-representation to the dimension of the exterior algebra $\bigwedge\left(\mathbb{C}^{2n}\otimes\mathbb{C}^{k}\right)$, was expressed as a determinantal point process governed by a correlation kernel derived from semiclassical orthogonal polynomials. Unlike the $\GL_{n}$ case, where free-fermionic methods apply, the symplectic setting required a Christoffel transformation of Krawtchouk polynomials to construct the relevant orthogonal polynomials. It is an open question how to apply the free-fermionic representation \cite{Baker_1996} to the skew $\Sp_{2n} \times \Sp_{2k}$ case. In particular, it is not clear how the tau functions for $d$-connections can be constructed \cite{10.1215/S0012-7094-06-13433-6,arinkin2009tau}. 

By analyzing the integral representation of these polynomials in the double-scaling limit $n, k \to \infty$ with $n/k$ fixed, we have demonstrated that the correlation kernel converges to the discrete sine kernel in the bulk regime. This universality mirrors results for general linear groups \cite{Borodin_2003, Borodin_2017} but required novel techniques due to the absence of free-fermionic structures.

Our findings align with the limit shape results in \cite{nazarov2024skew} and extend the understanding of fluctuations to symplectic symmetry classes.

The following open problems and directions naturally follow from the present work:
\begin{enumerate}
  \item \textbf{Free-fermionic representation for the skew $\Sp_{2n}\times\Sp_{2k}$ case.} Clarify whether and how the free-fermionic formalism used in the $\GL_{n}$ setting can be adapted to the symplectic skew Howe duality. See \cite{Baker_1996} for related representations.

  \item \textbf{Tau functions for $d$-connections.} Construct the appropriate tau functions in the symplectic skew-duality setup and explain their relation to the determinantal kernel obtained here \cite{10.1215/S0012-7094-06-13433-6,arinkin2009tau}.

  \item \textbf{Edge asymptotics and connections to discrete Painlev\'e equations.} Study edge scaling regimes (soft and hard edges) to determine whether discrete analogues of Airy/Tracy--Widom laws or Painlev\'e{}-type distributions arise, and relate these asymptotics to discrete Painlev\'e equations where applicable.
\end{enumerate}
\section{Appendix}
The goal of this appendix is to prove the relation for the second term in the definition (\ref{eq:orthonormal-BC-Christoffel-transform-explicit}).
\begin{lemma} Polynomials \(\widetilde{K}_m(x)\), defined by \eqref{shifted-monic-krawtchouk-hyp-rep}, satisfy the following relation:
\begin{align}\notag
\frac{\widetilde{K}_{m+2}(0)}{\widetilde{K}_{m}(0)}&=-\frac{(K-m)(m+1)}{4n^2}=\\ \label{eq:p1_2_term}
&=\frac{(m+1) (m+2) \binom{K}{m+2} \, _2F_1\left(-\frac{K}{2},-m-2;-K;2\right)}{4 n^2 \binom{K}{m} \, _2F_1\left(-\frac{K}{2},-m;-K;2\right)}. 
\end{align}
\end{lemma}
\begin{proof}
First of all we need to use the following transformation for $n\in \mathbb{N}$ (see (3.19) in \cite{Hannah2013IdentitiesFT})
\begin{align}
  \,_2F_1(-n,b,c,z)=\frac{(b)_n}{(c)_n}(-z)^n \,_2F_1(-n,1-c-n,1-b-n,1/z).
  \label{Hyp_z_inverse}
\end{align}
So we get
\begin{align*}
	&\,_2F_1(-m-2,-K/2,-K,2)=\\
	&=\frac{(-K/2)_{m+2}}{(-K)_{m+2}}(-1)^{m+2} 2^{m+2}\,_2F_1(-m-2,K-m-1,k/2-m-1,1/2).
\end{align*}
Now we can use the standard series for the Gauss hypergeometric function
\begin{align*}
	&\,_2F_1(-m-2,K-m-1,k/2-m-1,1/2)=\sum\limits_{i=0}^{m+2}\frac{(-1)^i}{2^i}\binom{m+2}{i}\frac{(K-m-1)_i}{(K/2-m-1)_i}=\\
	&=\frac{\left((-1)^m+1\right) \Gamma \left(\frac{m+3}{2}\right) \Gamma \left(\frac{1}{2} (-K+m+2)\right)}{2
   \sqrt{\pi } \Gamma \left(-\frac{K}{2}+m+2\right)}=\\
   &=-\frac{1/4(m+1)(K-m)}{(-K/2+m+1)(-K/2+m)}\,_2F_1(-m,K-m+1,K/2-m+1,1/2),
\end{align*} 
where the last sum can be computed explicitly by the Gauss's second summation theorem
\begin{align*}
	\,_2F_1(a,b,1/2(a+b+1),1/2)=\frac{\Gamma(1/2)\Gamma(1/2(a+b+1))}{\Gamma(1/2(1+a))\Gamma(1/2(1+b))}.
\end{align*}
Finally we have
\begin{align*}
	\,_2F_1(-m-2,-K/2,-K,2)=\frac{(m+1)}{(-K+m+1)}\,_2F_1(-m,-K/2,-K,2).
\end{align*}
Substituting this into the original expression, we immediately come to the desired proof.
\end{proof}~
\\\\\textbf{Aknowlegments} We thank Travis Scrimshaw for reading the first draft and pointing out the deficiencies in it. We thank Anton Dzhamay for useful discussions. 
\\\\\textbf{Funding} The study was carried out with the financial support of the Ministry of Science and Higher Education of the Russian Federation in the framework of a scientific project under agreement No. 075-15-2025-013.
\\\\\textbf{Data Availability} All data generated or analysed during this study are contained in this document.
\section*{Statements and Declarations}~\\
\textbf{Conflict of interest} On behalf of all authors, the corresponding author states that there is no conflict of
interest.

\bibliography{lit.bib}
 
\end{document}